\documentclass[11pt]{amsart}
\usepackage{amsmath,amssymb}
\usepackage{hyperref}
\usepackage{enumerate}
\usepackage{bm}
\usepackage{color}
  \usepackage{graphicx}
  \usepackage{latexsym} 
  \usepackage[all]{xy}  
  \usepackage{amsfonts} 
  \usepackage[2emode]{psfrag}

\newtheorem{proposition}{Proposition}[section]
\newtheorem{lemma}[proposition]{Lemma}
\newtheorem{corollary}[proposition]{Corollary}
\newtheorem{theorem}[proposition]{Theorem}

\theoremstyle{definition}
\newtheorem{definition}[proposition]{Definition}

\theoremstyle{remark}
\newtheorem{remark}[proposition]{Remark}

\newcommand{\thlabel}[1]{\label{th:#1}}
\newcommand{\thref}[1]{Theorem~\ref{th:#1}}
\newcommand{\selabel}[1]{\label{se:#1}}
\newcommand{\seref}[1]{Section~\ref{se:#1}}
\newcommand{\lelabel}[1]{\label{le:#1}}
\newcommand{\leref}[1]{Lemma~\ref{le:#1}}
\newcommand{\prlabel}[1]{\label{pr:#1}}
\newcommand{\prref}[1]{Proposition~\ref{pr:#1}}
\newcommand{\colabel}[1]{\label{co:#1}}
\newcommand{\coref}[1]{Corollary~\ref{co:#1}}
\newcommand{\relabel}[1]{\label{re:#1}}
\newcommand{\reref}[1]{Remark~\ref{re:#1}}

\newcommand{\eqlabel}[1]{\label{eq:#1}}
\newcommand{\equref}[1]{(\ref{eq:#1})}

\newcommand{\HOM}{{\sf HOM}}

\newcommand{\Vect}{{\sf Vect}}

\newcommand{\id}{{\sf id}\,}

\newcommand{\coev}{{\rm coev}}
\newcommand{\ev}{{\rm ev}}

\newcommand{\Hom}{{\sf Hom}}

\newcommand{\Rat}{{\sf Rat}}

\newcommand{\Fam}{{\sf Fam}}
\newcommand{\Maf}{{\sf Maf}}
\newcommand{\PFam}{{\sf PFam}}
\newcommand{\PMaf}{{\sf PMaf}}

\def\rff{\reflectbox{{\scriptsize \sf F}}}
\def\rss{\reflectbox{{\scriptsize \sf S}}}
\def\rpp{\reflectbox{{\scriptsize \sf P}}}
\def\rf{\reflectbox{\sf F}}
\def\rs{\reflectbox{{\sf S}}}
\def\rp{\reflectbox{{\sf P}}}

\def\Ab{\underline{\underline{\sf Ab}}}

\def\ot{\otimes}

\def\lie{{[-,-]}}

\def\id{\textrm{{\small 1}\normalsize\!\!1}}

\def\MM{{\mathbb M}}

\def\ZZ{{\mathbb Z}}

\def\lll{{\mathfrak l}}

\def\sss{{\mathfrak s}}

\newcommand{\Aa}{\mathcal{A}}
\newcommand{\Bb}{\mathcal{B}}
\newcommand{\Cc}{\mathcal{C}}
\newcommand{\Dd}{\mathcal{D}}

\newcommand{\Hh}{\mathcal{H}}

\newcommand{\Ll}{\mathcal{L}}
\newcommand{\Mm}{\mathcal{M}}

\newcommand{\Oo}{\mathcal{O}}

\newcommand{\Rr}{\mathcal{R}}

\newcommand{\Vv}{\mathcal{V}}

\def\text#1{{\rm {\rm #1}}}

\def\ol{\overline}
\def\ul{\underline}
\def\dul#1{\underline{\underline{#1}}}

\def\Set{\dul{\rm Set}}

\def\lim{{\rm lim\,}}

\def\u{{\rm u}}

\def\Alg{{\sf Alg}}

\def\CoAlg{{\sf CoAlg}}

\def\Bialg{{\sf BiAlg}}

\def\Hpfalg{{\sf HpfAlg}}

\def\LieAlg{{\sf LieAlg}}
\def\LieGrp{{\sf LieGrp}}

\def\LieCoAlg{{\sf LieCoAlg}}

\def\uoR{\underline{\overline{R}}}
\def\uoL{\underline{\overline{L}}}
\def\uoH{\underline{\overline{H}}}
\def\uR{\underline{R}}
\def\uL{\underline{L}}
\def\oR{\overline{R}}
\def\uH{\underline{H}}
\def\oH{\overline{H}}
\def\oL{\overline{L}}
\def\wR{\widetilde{R}}
\def\wL{\widetilde{L}}

\def\Fam{{\sf Fam}}

\begin{document}
\title[Duality of Lie and Hopf algebras]{On the duality of generalized Lie and
Hopf algebras}

\author{I.\ Goyvaerts}
\address{Department of Mathematics, Faculty of Engineering, Vrije Universiteit Brussel, Pleinlaan 2, B-1050 Brussel, Belgium}
\email{igoyvaer@vub.ac.be}
\author{J.\ Vercruysse}
\address{D\'epartement de Math\'ematiques, Facult\'e des sciences, Universit\'e Libre de Bruxelles, Boulevard du Triomphe, B-1050 Bruxelles, Belgium}
\email{jvercruy@ulb.ac.be}

\thanks{\textcolor{blue}{This version of the paper differs slightly from the published one (as appeared in {\sl Adv. Math.} \textbf{258} (2014), p. 154-190). After publication in the cited journal, the authors were made aware of missing conditions in \leref{functorsuo} and \thref{uoadjunction} for the properties mentioned in the statements to hold in full generality for Hopf algebras, too. The actual version of the article aims at providing a correction to this inaccuracy. We also point out that our main theorems still hold in full generality. Moreover, the additional conditions are fulfilled in the cases of interest, so that this correction does not harm any of the results of the paper.}}
\begin{abstract}
We show how, under certain conditions, an adjoint pair of braided monoidal functors can be lifted to an adjoint pair between categories of Hopf algebras. This leads us to an abstract version of Michaelis' theorem, stating that given a Hopf algebra $H$, there is a natural isomorphism of Lie algebras $Q(H)^*\cong P(H^\circ)$, where $Q(H)^*$ is the dual Lie algebra of the Lie coalgebra of indecomposables of $H$, and $P(H^\circ)$ is the Lie algebra of primitive elements of the Sweedler dual of $H$. We apply our theory to Turaev's Hopf group-(co)algebras. 
\end{abstract}
\maketitle

\tableofcontents

\section*{Introduction and motivation}

It is widely known that the theories of Lie algebras, Lie groups and Hopf algebras are strongly interrelated. More precisely, given a Lie group $G$, one can associate to it the Lie algebra $T_1(G)$ (the tangent space at unity) and the Hopf algebra $\Rr(G)$ (the algebra of representative functions on $G$), where we work over a fixed base field $k$ of (for example) the real numbers. 
Moreover, $T_1$ is in fact part of a pair of adjoint functors $(\Gamma,T_1)$ between finite-dimensional Lie algebras and Lie groups, that even becomes an equivalence of categories if we restrict to simply connected Lie groups. Furthermore, given a Lie algebra $L$, one can construct the Hopf algebra $U(L)$ (the universal enveloping algebra of $L$).
Given a Hopf algebra $H$, one can construct a new Hopf algebra $H^\circ$ (being the Sweedler dual of $H$) and the Lie algebra $P(H)$ (that consists of all primitive elements in $H$). It is known that $(-)^\circ$ is a contravariant self-adjoint endofunctor (see e.g.\ \cite{Abe}, Chapter 2, section 3.5, page 87) and $(U,P)$ is an adjoint pair of functors. We can summarize the above in the following diagram.
\[
\xymatrix{
\LieAlg \ar@<.5ex>[rr]^-U \ar@<.5ex>@{.>}[d]^{\Gamma} && \Hpfalg \ar@<.5ex>[ll]^-P \ar@/^/@<.5ex>[d]^{(-)^\circ}\\
\LieGrp \ar@/^/@<.5ex>[rr]^-{\Rr} \ar@<.5ex>[u]^{T_1} && \Hpfalg  
\ar@/^/@<.5ex>[u]^{(-)^\circ}
}
\]
Here the bended arrows are contravariant functors, the straight arrows are covariant (and the arrow of the functor $\Gamma$ is dotted because it only applies to finite-dimensional Lie algebras).

The above diagram has many interesting properties, for example in the case of the Lie group $SL_n(k)$ (working over a base field $k$ of characteristic $0$) with associated Lie algebra ${\sss\lll}_{n}(k)$, we have an isomorphism (see e.g.\ \cite{Mont:book}, Example 9.1.6)
\begin{equation}\eqlabel{DualMont}U({\sss\lll}_{n}(k))^\circ\cong \Oo(SL_n(k)),\end{equation}
where $\Oo(SL_n(k))$ is the subalgebra of $\Rr(SL_n(k))$ consisting of polynomial functions on $SL_n(k)$.

Unfortunately, the above isomorphism cannot be extended to a natural isomorphism valid for all Lie algebras or Lie groups. Michaelis proposed an alternative approach, replacing Lie groups by so-called Lie coalgebras. A Lie coalgebra is a formal dualization of a Lie algebra, in the same sense that a $k$-coalgebra dualizes a $k$-algebra. More precisely, a Lie coalgebra is a vector space $C$ endowed with an operation $\Upsilon:C\to C\ot C$ that is antisymmetric and satisfies a dual version of the Jacobi identity. This leads to the following duality diagram 
\[
\xymatrix{
\LieAlg \ar@<.5ex>[rr]^-U \ar@/^/@<.5ex>[d]^{(-)^\circ} && \Hpfalg \ar@<.5ex>[ll]^-P \ar@/^/@<.5ex>[d]^{(-)^\circ}\\
\LieCoAlg \ar@<.5ex>[rr]^-{U^c} \ar@/^/@<.5ex>[u]^{(-)^*} && \Hpfalg \ar@<.5ex>[ll]^{Q} \ar@/^/@<.5ex>[u]^{(-)^\circ}
}
\]
In this diagram, the functors $Q$ and $U^c$ are constructed in a dual way to the classical functors $P$ and $U$ (see \cite{Mich} for more details of this construction). The functor $(-)^*$ takes the underlying vector space of a Lie coalgebra to its dual vector space, which can be endowed in a canonical way with a Lie algebra structure. This is obviously a contravariant functor. As for usual algebras (and Hopf algebras), for a Lie algebra $L$ with bracket $[-,-]$, the space $L^\circ$ is a subspace of $L^*$ consisting of those functionals $f\in L^*$ such that $f\circ \lie\in L^*\ot L^*$. Remark that the new diagram here above looks more symmetric: all horizontal functors are covariant, while vertical functors are contravariant. Moreover, there is no longer a finiteness restriction on Lie algebras (or Lie coalgebras) and the diagram is now completed with a functor from Hopf algebras to Lie coalgebras. Michaelis proved that given a Lie algebra $L$, the diagram commutes in the sense that there is a natural isomorphism of Hopf algebras
$$U(L)^\circ \cong U^c(L^\circ).$$
In some sense, this is the Lie coalgebra version of the desired duality result mentioned above. Furthermore, given a Hopf algebra $H$, Michaelis \cite{Mich2} proved that the diagram also commutes in the sense that we have a natural isomorphism of Lie algebras
\begin{equation}\eqlabel{Mich}Q(H)^*\cong P(H^\circ).\end{equation}

Despite this remarkable fact, the Sweedler dual of a Hopf algebra can become trivial (isomorphic to the base field), even for a non-trivial Hopf algebra (see \cite{BCM}, 2.7 e.g. for a concrete classical example). 
Hence, 
to preserve non-trivial but similar results as \equref{DualMont} or \equref{Mich}, one is forced to consider other dualities than the Sweedler dual in these situations. With the rise of quantum groups, a non-commutative Pontryagin duality has been developed by Van Daele (\cite{VD:Mult}, \cite{VD:duality}), Kustermans and Vaes (\cite{KusVae:LCQG}), amongst others, allowing a pair of self-adjoint endofunctors on a category of generalized Hopf algebras, such as multiplier Hopf algebras and locally compact quantum groups.

The aim of the presented paper is to extend the results of Michaelis to a more abstract setting that allows, motivated by the considerations above, also an application to other dualities than the Sweedler dual. 

More precisely, given a braided monoidal functor $R:\Cc\to \Dd$ with a (hence braided monoidal) left adjoint $L$, it is known that $R$ can be lifted to a monoidal functor $\oR:\Alg(\Cc)\to\Alg(\Dd)$ between the categories of algebras in $\Cc$ and $\Dd$ respectively. Similarly, $L$ lifts to a functor $\uL$ between the categories of coalgebras in $\Dd$ and $\Cc$ respectively. We will say that the pair $(L,R)$ is liftable if $\oR$ and $\uL$ allow a left, respectively right, adjoint (denoted by $\oL$ and $\uR$ respectively). We show that a liftable pair can be lifted to an adjoint pair of functors between the categories of {\color{blue}{bialgebras in $\Cc$ and $\Dd$. The lifted pair determines an adjunction between the categories of Hopf algebras in $\Cc$ and $\Dd$ provided two extra conditions hold}}. Suppose now that the categories $\Cc$ and $\Dd$ are additive symmetric monoidal, and $R$ is an additive symmetric monoidal functor. Then $R$ can also be lifted to a functor $\wR:\LieAlg(\Cc)\to\LieAlg(\Dd)$ between the categories of Lie algebras in $\Cc$ and $\Dd$ respectively. Moreover, $\Ll_\Dd\oR=\wR\Ll_\Cc$, where a functor $\Ll_\Aa:\Alg(\Aa)\to\LieAlg(\Aa)$ computes the commutator Lie algebra of an algebra in an additive symmetric monoidal category $\Aa$. We say that $R$ is Lie-liftable if the left adjoint $L$ of $R$ also lifts to a left adjoint $\wL$ for $\wR$. For a pair of liftable, Lie-liftable functors $(L,R)$ between symmetric monoidal categories $\Cc$ and $\Dd$ \textcolor{blue}{satisfying the appropriate conditions}, we then obtain a diagram of (all covariant) adjoint functors 
\[
\xymatrix{
\Hpfalg(\Cc) \ar@<.5ex>[rr]^-{\uoR} \ar@<.5ex>[d]^{P} && \Hpfalg(\Dd) \ar@<.5ex>[d]^{P} \ar@<.5ex>[ll]^-{\uoL} \\
\LieAlg(\Cc) \ar@<.5ex>[u]^{U} \ar@<.5ex>[rr]^{\wR} && \LieAlg(\Dd) \ar@<.5ex>[u]^{U} \ar@<.5ex>[ll]^-\wL
}
\]
For any Lie algebra $K$ in $\Dd$, there is a natural isomorphism of Hopf algebras in $\Cc$
$$\uoL UK\cong U\wL K$$
and for any Hopf algebra $H$ in $\Cc$, then there is a natural isomorphism of Lie algebras in $\Dd$
$$P\uoR H\cong \wR P H.$$
All this will be done in \seref{liftings}.

We then apply the obtained result to some concrete situations. First, we recover the result of Michaelis cited above in \seref{Mich}, by taking $\Dd=\Vect$, $\Cc=\Vect^{op}$ and $R$ the functor that computes the dual vector space. The main application is given in \seref{Tur}, where we consider the category of families of finite-dimensional vector spaces and families of the opposite of finite-dimensional vector spaces. The canonical isomorphism between a finite-dimensional vector space and its dual, induces a pair of functors between these categories upon which our theory applies. In this way, we obtain a version of Michaelis' theorem for Turaev's Hopf group-coalgebras, which play an important role in Homotopy Quantum Field Theories \cite{Tur} (see also \cite{Vir}). Indeed, as it was observed in \cite{CaeDel}, Hopf group-coalgebras can be treated as Hopf algebras in the symmetric monoidal category of families of dual vector spaces (called the Turaev category in the cited paper). Moreover, it was shown in \cite{ADVanDaele} (see also \cite{JV:KHA}) that these Hopf algebras provide an (easy) class of examples of multiplier Hopf algebras (discrete quantum groups) and therefore can serve as a toy-model for more complicated cases (of locally compact quantum groups) and (their) dualities, a direction for future work.

\section{Preliminaries}

\subsection*{Setting and notation} 
Throughout, we will work in the setting of monoidal categories. Without restrictions on our results, we will always assume that a monoidal category is strict. Moreover, we consider braided, symmetric and additive monoidal categories. A basic reference for these notions is e.g.\ \cite{Mac}. 

When $X$ is an object in a category $\Cc$, we will denote the identity morphism on $X$ by $1_X$ or $X$ for short.

Troughout the paper, $k$ is a field of characteristic $0$ and ${\Vect}$ denotes the category of $k$-vector spaces and $k$-linear maps.

\subsection*{Monoidal functors}
A functor $F:\Cc\to\Dd$ between monoidal categories $(\Cc,\ot,I)$ and $(\Dd,\odot,J)$ is said to be a monoidal functor if it comes equipped with a family of natural morphisms $\phi_{X,Y}:F(X)\odot F(Y)\to F(X\ot Y)$, $X,Y\in \Cc$ and
a $\Dd$-morphism $\phi_0:J\to F(I)$, satisfying the known suitable compatibility conditions with respect to the associativity and unit constraints of $\Cc$ and $\Dd$. Moreover, $(F,\phi_0,\phi)$ is called a {\em strong} monoidal functor if $\phi_0$ is an isomorphism and $\phi_{X,Y}$ is a natural isomorphism for any objects $X,Y\in\Cc$. $(F,\phi_0,\phi)$ is called a {\em strict} monoidal functor $(F,\phi_0,\phi)$ if $\phi_0$ is the identity morphism and $\phi$ is the identity natural transformation.

Dually, an {\em op-monoidal functor} $F:\Cc\to \Dd$ is a functor for which there exists a morphism $\psi_0:F(I)\to J$ in $\Dd$   and morphisms $\psi_{X,Y}:F(X\ot Y)\to F(X)\odot F(Y)$ in $\Dd$, that are natural in $X, Y\in \Cc$, satisfying suitable compatibility conditions.
Any strong monoidal functor $(F,\phi_0,\phi)$ is automatically op-monoidal: it suffices to take $\psi_0=\phi_0^{-1}$ and $\psi=\phi^{-1}$.

If $\Cc$ and $\Dd$ are braided monoidal categories, then a {\em braided monoidal functor} $F:\Cc\to \Dd$ is a monoidal functor such that $F\gamma_{X,Y}\circ \phi_{X,Y}=\phi_{Y,X}\circ \gamma_{FX,FY}:FX\odot FY\to F(Y\ot X)$, where $\gamma$ denotes the braiding in $\Cc$ and $\Dd$.

A {\em monoidal natural transformation} between (braided) monoidal functors $(F,\phi,\phi_0)$ and $(G,\psi,\psi_0)$ is a natural transformation $\alpha$ that satisfies $\alpha_{X\ot Y}\circ \phi_{X,Y}=(\alpha_X\odot\alpha_Y)\circ \psi_{X,Y}$ and $\phi_0\circ \alpha_I=\psi_0$.

\subsection*{Rigidity}
An object $X$ in a monoidal category is called {\em left rigid} if there exists an object $X^\dagger$ together with morphisms $\coev:I\to X\ot X^\dagger$ and $\ev:X^\dagger\ot X\to I$ such that 
$$(X\ot\ev)\circ (\coev\ot X)=X, \quad (\ev\ot X^\dagger)\circ (X^\dagger\ot\coev)= X^\dagger$$
It is easily verified that if $X$ is left rigid, then the object $X^\dagger$ is unique up to isomorphism. In this situation, we call $X^\dagger$ the {\em left dual} of $X$ and $(X^\dagger,X,\ev,\coev)$ an {\em adjoint pair} in $\Cc$.

A {\em right rigid} object is defined symmetrically. Remark that if $X$ is left rigid with left dual $X^\dagger$, then $X^\dagger$ is right rigid with right dual $X$.
A monoidal category is said to be {\em left rigid} (resp. right rigid, resp. rigid) if every object is left (resp. right, resp. both left and right) rigid. Another name for a rigid monoidal category is an {\em autonomous (monoidal) category}. If $\Cc$ is braided, then it is right rigid if and only if it is left rigid, in this situation the left and right duals of an object are isomorphic and we speak just about the dual for short.

\subsection*{Contravariant adjunctions}

In most of our applications, the functors will be contravariant. Let us make a short remark about contravariant adjunctions. If $(L:\Dd\to \Cc,R:\Cc\to \Dd)$ is an adjoint pair of contravariant functors, then we can consider them as covariant functors, by considering the opposite structure on $\Cc^{op}$. That is, $(L:\Dd\to \Cc^{op},R:\Cc^{op}\to \Dd)$ is an adjoint pair in the usual sense. If we use this convention (remark that one could as well have chosen an opposite structure on the category $\Dd$), the adjunctional isomorphisms are
$$\Hom_{\Cc^{op}}(LD,C)=\Hom_{\Cc}(C,LD)\cong \Hom_{\Dd}(D,RC).$$
Consequently, a contravariant adjunction $(L,R)$ is characterized by two units
\begin{eqnarray}
\nu_C:&&C\to LRC\\
\zeta_D:&&D\to RLD
\end{eqnarray}
satisfying 
\begin{eqnarray}
R\nu_C\circ \zeta_{RC}&=& RC\\
L\zeta_D\circ \nu_{LD}&=& LD
\end{eqnarray}
Remark that by the contravariance of $R$ and $L$, we have $R\nu_C:RLRC\to RC$ in $\Dd$ and $L\zeta_D:LRLD\to LD$ in $\Cc$.

\subsection*{Algebraic objects in monoidal categories}
Let $\Cc=(\Cc,\ot,l)$ be a (strict) monoidal category. An {\em algebra} in $\Cc$ is a triple
$A=(A,m,u)$, where $A\in \Cc$ and $m: A\ot A\to A$ and $\u: I\to A$ are
morphisms in $\Cc$ such that the following diagrams commute:
$$
\xymatrix{
A\ot A\ot A\ar[rr]^{m\ot A}\ar[d]_{A\ot m}&& A\ot A\ar[d]^{m}\\
A\ot A\ar[rr]^{m}&&A}~~~~~~
\xymatrix{I\ot A\ar[r]^{u\ot A}\ar[rd]_{\cong}&
A\ot A\ar[d]^{m}&A\ot I\ar[l]_{A\ot u}\ar[dl]^{\cong}\\
&A&}
$$
A {\em coalgebra} in a monoidal category $\Cc$ is defined as an algebra in the opposite category $\Cc^{\rm op}$. Here, $\Cc^{op}=(\Cc^{op},\ot^{op},I)$ denotes the opposite category of $\Cc$ and $\ot^{op}: \Cc^{op}\times\Cc^{op}\to \Cc^{op}$ the opposite tensor product functor induced in the obvious way by $\ot$. Algebras and coalgebras in the monoidal category of modules over a commutative base ring $k$ are called $k$-algebras and $k$-coalgebras respectively.
 
Now, consider $\Cc$ to be moreover braided and denote by $c$ its braiding. A {\em bialgebra} in $\Cc$ is a pentuple
$H=(H,m,u,\Delta,\varepsilon)$, where $(H,m,u)$ is an algebra and 
$(H,\Delta,\varepsilon)$ a coalgebra in $\Cc$, such that $\Delta:H\to H\ot H$ and $\varepsilon: H\to I$ are algebra morphisms, or, equivalently, such that $m:H\ot H\to H$ and $u:I\to H$ are coalgebra morphisms. Here the algebra structure on $H\ot H$ is given by 
\begin{equation}
\xymatrix{H\ot H\ot H\ot H\ar[rr]^-{H\ot c_{H,H}\ot H} && H\ot H\ot H\ot H\ar[r]^-{m\ot m} & H\ot H},
\end{equation}
and the coalgebra structure by 
\begin{equation}
\xymatrix{H\ot H\ar[r]^-{\Delta\ot \Delta} & H\ot H\ot H\ot H\ar[rr]^-{H\ot c_{H,H}\ot H}&& H\ot H\ot H\ot H}.
\end{equation}
Let $C=(C,\Delta,\varepsilon)$ be a coalgebra, and $A=(A,m,u)$ an algebra in $\Cc$. We can define
a product $*$ on $\Hom_\Cc(C,A)$ as follows: for $f,g: C\to A$, let
$f*g$ be defined by
$f*g = m\circ(f\ot g)\circ\Delta$.
$*$ is called the {\it convolution product}. In this way, $\Hom_\Cc(C,A)$ becomes an associative monoid with product $*$ and unit $u_A\circ \epsilon_C$. In particular, if $C$ is a $k$-coalgebra, then the dual space $C^*=\Hom(C,k)$ is a $k$-algebra.
Now suppose that $H$ is a bialgebra in $\Cc$, and take $H=A=C$ in the above construction.
If the identity $H$ of $H$ has a convolution inverse $S$, then we
say that $H$ is a {\em Hopf algebra} (or {\em Hopf monoid}) in $\Cc$. In this case, $S$ is called the 
{\em antipode} of $H$.
\\Symmetric monoidal categories are well-suited to consider cyclic permutations on copies of the same object. Moreover, if we assume additivity, these categories offer a framework to generalize the classical definition of Lie algebra and incorporate examples of ``generalized'' Lie algebras. So, using same notation as above, let $\Cc$ be an additive, symmetric monoidal category.
A \textit{Lie algebra in} $\Cc$ is a couple $(L,[-,-])$, denoted by $L$ for short if there is no confusion possible, where $L$ is an object of $\Cc$ and $[-,-]:\ L\ot L\to L$ is a morphism (which we call a {\em Lie bracket}) in $\Cc$ that satisfies
\begin{eqnarray}\eqlabel{ASsymm}
[-,-] \circ (1_{L\ot L} + c_{L,L})&=&0_{L\ot L,L},\\
\eqlabel{Jacsymm}
[-,-]\circ (1_{L}\ot [-,-])\circ (1_{L\ot(L\ot L)}+ t_c+ w_c)&=&0_{L\ot (L\ot L),L},
\end{eqnarray}
where $t_c=(c\ot L)\circ(L\ot c)$ and $w_c=(L\ot c)\circ (c\ot L)$.
A {\em Lie coalgebra} in $\Cc$ is a Lie algebra in $\Cc^{op}$. Explicitly, this means that a Lie coalgebra is a pair $(C,\Upsilon)$, where $\Upsilon :C\to C\ot C$ is a map, called the \textit{Lie co-bracket}, that satisfies the following two conditions
\begin{eqnarray*}
(1_{C\ot C} + c_{C,C})\circ \Upsilon &=&0_{C,C\ot C};\\
(1_{C\ot(C\ot C)}+ t_c+ w_c)\circ (1_{C}\ot \Upsilon)\circ \Upsilon&=&0_{C,C\ot (C\ot C)}.
\end{eqnarray*}

Recall that an algebra $(A,m,u)$ in a monoidal category $\Cc$ is called a {\em Frobenius algebra} if it possesses at the same time a coalgebra structure $(A,\Delta,\epsilon)$ rendering commutative the following two diagrams.
\[
\xymatrix{
A\ot A \ar[rr]^-{A\ot \Delta} \ar[d]_-{m} && A\ot A\ot A \ar[d]^-{m\ot A} \\
A\ar[rr]^-{\Delta} && A\ot A
}\quad
\xymatrix{
A\ot A \ar[rr]^-{\Delta\ot A} \ar[d]_-{m} && A\ot A\ot A \ar[d]^-{A\ot m} \\
A\ar[rr]^-\Delta && A\ot A
}
\]
In this situation, we also call $(A,\Delta,\epsilon)$ a Frobenius coalgebra.
Equivalently, an algebra $A$ is Frobenius if it is a rigid object in $\Cc$ and there is an isomorphism between the regular (right) $A$-module (i.e. $A$ with action $m$) and the dual (right) $A$-module (i.e. $A^\dagger$ with right action $(\ev\ot A^\dagger)\circ(A^\dagger\ot m\ot A^\dagger)\circ(A^\dagger\ot A\ot \coev)$).

In particular, working over a base field $k$, a Frobenius $k$-(co)algebra $A$ is always finite-dimensional. More generally, a $k$-coalgebra is called 
(left) {\em co-Frobenius} if there exists an injective (left) $C^*$-module morphism $C\to C^*$. Recall that the left {\em rational} part of $C^*$ is the set
$$\Rat({_{C^*}}C^*)=\{f\in C^*~|~\forall g\in C^*, \exists f_1,\ldots f_n\in C^*, c_1,\ldots,c_n\in C,\ s.t.\ g*f=\sum_{i=1}^ng(c_i)f_i\}$$
Recall that the rational part becomes a right $C$-comodule with coaction $\rho(f)=\sum_{i=1}^nf_i\ot c_i=f_{[0]}\ot f_{[1]}$.
Suppose now that $H$ is a Hopf algebra. A left {\em integral} on $H$ is an element $t\in H^*$ such that $f*t=f(1)t$ for all $f\in H^*$. The space of all left integrals is denoted as $\int_l$.
If $H$ is a Hopf algebra, then it is known that $H$ is left co-Frobenius if and only if $H$ is right co-Frobenius if and only there exists a left (or right) integral on $H$ if and only if the left (or right) rational part of $H^*$ is non-zero. In this situation, the map 
$$\int_l\ot H\to\Rat(_{H^*}H^*)~;~ t\ot h\mapsto t(-S(h))$$
is an isomorphism.

\section{Liftings of dualities: An abstract version of Michaelis' theorem}

In this Section, we generalize Michaelis' theorem \cite{Mich2}, stating that given a $k$-Hopf algebra $H$, there is a canonical isomorphism of $k$-Lie algebras
$$P(H^{\circ})\cong Q(H)^*.$$

\subsection{Lifting of an adjunction to Hopf algebras}\selabel{liftings}

In this Section, we will use the following classical result. Without repeating the complete proof in this note, let us recall the construction, as our computations depend on it.

\begin{theorem}\thlabel{monoidal&adjoint}
Let $(\Cc,\ot,I)$ and $(\Dd,\odot,J)$ be (braided) monoidal categories, and $(L,R)$ a pair of adjoint functors between $\Cc$ and $\Dd$ with unit $\alpha:\id_\Dd\to RL$ and counit $\beta:LR\to \id_\Cc$.
\[
\xymatrix{
\Cc \ar@<.5ex>[rr]^-R && \Dd \ar@<.5ex>[ll]^-L
}
\]
Then there is a bijective correspondence between (braided) monoidal structures $(\phi,\phi_0)$ on $R$ and (braided) op-monoidal structures $(\psi,\psi_0)$ on $L$. Furthermore the following diagrams commute canonically for all $C,C'\in \Cc$ and $D,D'\in\Dd$
\begin{eqnarray*}
\xymatrix{
 D\odot D' \ar[d]_{\alpha_{D\odot D'}} \ar[rr]^{\alpha_D\odot \alpha_{D'}}&& RLD\odot RLD' \ar[d]^{\phi_{LD,LD'}} \\
RL(D\odot D') \ar[rr]_{R\psi_{D,D'}} 
&& R(LD\ot LD')
}&&
\xymatrix{
L(RC\odot RC') \ar[d]_{L\phi_{C,C'}} \ar[rr]^{\psi_{RC,RC'}}&& LRC\ot LRC' \ar[d]^{\beta_C\ot \beta_{C'}}\\
LR(C\ot C') \ar[rr]_{\beta_{C\ot C'}}  
&& C\ot C'
}
\end{eqnarray*}
\end{theorem}
\begin{proof}
Suppose that $R$ is a monoidal functor $(R,\phi,\phi_0)$. Then one equips $L$ with an op-monoidal structure by putting
\begin{eqnarray*}
\psi_{D,D'}&:&\xymatrix{L(D\odot D')\ar[r]^-{L\alpha_{D}\odot \alpha_{D'}} & L(RLD\odot RLD') \ar[r]^-{L\phi_{LD,LD'}}
& LR(LD\ot LD') \ar[r]^-{\beta_{LD\ot LD'}} & LD\ot LD'}
\\
\psi_0&:&\xymatrix{LJ\ar[rr]^-{L\phi_0} && LRI \ar[rr]^-{\beta_I} && I}
\end{eqnarray*}
If $R$ is moreover braided monoidal, then $L$ is braided op-monoidal with the $\psi_{-,-}$ and $\psi_{0}$ as above. Conversely, if $(L,\psi,\psi_0)$ is (braided) op-monoidal, than $R$ becomes monoidal by putting $\phi_{C,C'}=R(\beta_C\odot\beta_{C'})\circ R\phi_{RC,RC'}\circ \alpha_{RC\ot RC'}$ and $\phi_0=R\psi_0\circ \alpha_J$. The last statement follows by formulas above and the unit-counit property of the adjunction.
\end{proof}

We denote the category of algebras in the monoidal category $(\Cc,\ot,I)$ by $\Alg(\Cc)$. It consists of algebras in $\Cc$ as objects. A morphism $f:\ul A\to \ul B$ in $\Alg(\Cc)$ is a $\Cc$-morphism $f:A\to B$ such that $u_B=f\circ u_A$ and $m_B\circ (f\ot f)=f\circ m_A$.

If $\Cc$ is moreover a braided monoidal category with braiding $\gamma$, then $\Alg(\Cc)$ is a monoidal category with monoidal product
$\ul A\ot \ul B=(A\ot B,(m_A\ot m_B)\circ (A\ot \gamma_{B,A}\ot B), u_A\ot u_B)$. The monoidal unit in $\Alg(\Cc)$ is the trivial algebra on the unit object $I$.
Observe that we have a strict monoidal (forgetful) functor 
$$U_\Cc:\Alg(\Cc)\to \Cc.$$
The following result is again classical, we only state formulas of the construction, as these are useful for our later computations.
\begin{theorem}\thlabel{liftingalgebras}
Let $(\Cc,\ot,I,\gamma)$ and $(\Dd,\odot,J,\chi)$ be braided monoidal categories and $R:\Cc\to \Dd$ a braided monoidal functor. Then $R$ induces a monoidal functor $(\oR,\ol\phi,\ol\phi_0)$ such that the following diagram of monoidal functors commutes
\[
\xymatrix{
\Alg(\Cc) \ar[rr]^-{\oR} \ar[d]^{U_\Cc} && \Alg(\Dd) \ar[d]^{U_\Dd} \\
\Cc \ar[rr]^R && \Dd
}
\]
i.e.\ $R U_\Cc=U_\Dd\oR$, where $U_\Cc$ and $U_\Dd$ denote the canonical forgetful functors.
\end{theorem}

\begin{proof}
For an algebra $\ul A=(A,m_A,u_A)$ in $\Alg(\Cc)$, we define $\ol R\ul A=(RA,Rm_A\circ\phi_{A,A},Ru_A\circ \phi_0)$. For a morphism $f$ in $\Alg(\Cc)$, we define $\oR f=Rf$. 

To recover the monoidal structure $(\ol\phi,\ol\phi_0)$ on $\oR$, one can verify that for two algebras $\ul A$ and $\ul B$, the map $\phi_{A,B}$ is a morphism of algebras between $\oR\ul A\odot\oR\ul B$ and $\oR(\ul A\ot\ul B)$. Hence we can define 
$\ol\phi_{\ul A,\ul B}= \phi_{A,B}$. Similarly, $\phi_0$ is an algebra morphism from $J\to \oR I$, and we define $\ol\phi_0=\phi_0$.
\end{proof}

Dually to the above, the category $\CoAlg(\Dd)$ (resp. $\CoAlg(\Cc)$) of coalgebras in $\Dd$ (resp. in $\Cc$) is monoidal and a braided op-monoidal functor $L:\Dd\to\Cc$ induces an op-monoidal functor $\uL$ 
\[
\xymatrix{
\CoAlg(\Cc)\ar[d]^{U^\Cc} && \CoAlg(\Dd) \ar[ll]^-{\uL} \ar[d]^{U^\Dd} \\
\Cc  && \Dd \ar[ll]^L
}
\]

\begin{remark}\relabel{liftable}
Let $(L,R)$ be an adjoint pair between braided monoidal categories such that $R$ is a monoidal functor, and use notation as above.
In this note, we will say that the pair $(L,R)$ is {\em liftable} if the functors $\oR$ and $\uL$ allow a left, respectively right adjoint that we denote by $\oL$ and $\uR$ respectively.
\end{remark}

Our aim is now to prove that a liftable pair of functors induces an adjoint pair of functors between the associated categories of \textcolor{blue}{bialgebras. We will obtain this result in different steps. 
\\In order for the lifted adjunction to hold at the Hopf algebra level, two extra conditions are needed (see \leref{functorsuo}). To state these conditions, remark that, whenever $\oL$ exists, one has a natural transformation $\kappa: L\circ U_\Dd\to U_\Cc \circ \oL$ characterized by commutativity of the following diagram (here, $\ul B$ is an algebra in $\Dd$):}
\begin{equation}\eqlabel{kappa}
\xymatrix{
U_\Dd \ul B \ar[rr]^{\eta_{U_\Dd \ul B}}\ar[rrd]_{U_\Dd \overline{\eta}{\ul B}}&& {RL U_\Dd \ul B} \ar[d]^{R\kappa_{\ul B}}\\
&& U_\Dd \oR\oL \ul B = RU_\Cc \oL \ul B}
\end{equation}
\textcolor{blue}{
Similarly, if $\uR$ exists, one has a natural transformation $\kappa': R\circ U^\Cc\to U^{\Dd} \circ \uR$.}
\begin{lemma}\lelabel{functorsuo}
Let $(L,R)$ be a liftable pair of functors between braided monoidal categories and use notation as above. Then there exists a functor 
$$\uoL:\Bialg(\Dd)\to\Bialg(\Cc)$$
commuting with the forgetful functor to algebras and a functor 
$$\uoR:\Bialg(\Cc)\to\Bialg(\Dd)$$
commuting with the forgetful functor to coalgebras. 
\\The same result holds, replacing bialgebras by Hopf algebras, \color{blue}{provided $\kappa_{\ol H'}$ is an epimorphism in $\Cc$ and $\kappa'_{\ul H}$ is an epimorphism in $\Dd$ whenever $\uoH$ is a Hopf algebra in $\Cc$ and $\uoH'$ is a Hopf algebra in $\Dd$}.
\end{lemma}

\begin{proof}
Recall that $\Bialg(\Cc)=\Alg(\CoAlg(\Cc))=\CoAlg(\Alg(\Cc))$. Since the functor $\ol R:\Alg(\Cc)\to\Alg(\Dd)$ is monoidal, its left adjoint $\ol L:\Alg(\Dd)\to \Alg(\Cc)$ is op-monoidal. Hence it induces a functor $\uoL:\Bialg(\Cc)=\CoAlg(\Alg(\Cc))\to \CoAlg(\Alg(\Dd))=\Bialg(\Dd)$, which commutes with the forgetful functors. Explicitly, we have for any $\uoH'\in \Bialg(\Dd)=\CoAlg(\Alg(\Dd))$, $\uoH'=(\uH',\Delta',\epsilon')$, where $\uH'$ is the underlying algebra and $\Delta'$ and $\epsilon'$ are algebra morphisms, 
$$\uoL \uoH'=(\ol L \uH', \ol\psi_{\uH',\uH'}\circ \ol L\Delta', \ol \psi_0\circ \ol L\epsilon')$$
Similarly, for $\uoH\in \Bialg(\Cc)=\Alg(\CoAlg(\Cc))$, $\uoH=(\oH,m,u)$, where $\oH$ is the underlying coalgebra and $m$ and $u$ are algebra morphisms, 
$$\uoR \uoH=(\ul R \oH, \ul Rm\circ \ul \phi_{\oH,\oH}, \ul R u\circ \ul\phi_0)$$
Now suppose that $\uoH$ is a Hopf algebra in $\Cc$, i.e. $\uoH$ is a bialgebra enriched with an antipode morphism $S:H\to H$. Since an antipode is a coalgebra morphism $\oH\to \oH^{cop}$, where $\oH^{cop}$ is the co-opposite coalgebra, we can compute $\uR S:\uR \oH\to \uR \oH^{cop}$. \textcolor{blue}{One checks (see \cite[Proposition 31]{PS}) that $\uR S$ is indeed an antipode for the bialgebra $\uoR \uoH$ provided $\kappa'_{\ul H}$ is an epimorphism in $\Dd$. Similarly, the functor $\uoL$ sends Hopf algebras in $\Dd$ to Hopf algebras in $\Cc$ provided $\kappa_{\ol H'}$ is an epimorphism in $\Cc$ whenever $\uoH'$ is a Hopf algebra in $\Dd$.}
\end{proof}

\begin{lemma}\lelabel{dualcoalgebramorphism}
Let $(L,R)$ be a pair of adjoint functors between monoidal categories $\Cc$ and $\Dd$ such that $L$ is an op-monoidal functor (equivalently, $R$ is monoidal). Consider a coalgebra $\ol C=(C,\Delta,\epsilon)$ in $\Cc$ and a coalgebra $\ol D=(D,\Delta',\epsilon')$ in $\Dd$. Then the natural isomorphism
$$\theta_{C,D}:\Hom_\Cc(LD,C)\to \Hom_\Dd(D,RC)$$
induces a natural isomorphism between $\Hom_{\CoAlg(\Cc)}(\uL\ol D,\ol C)$ and the set of all $f\in\Hom_\Dd(D,RC)$ that satisfy
\begin{eqnarray}
R\Delta \circ f&=&\phi_{C,C}\circ (f\odot f)\circ \Delta' \eqlabel{f1}\\
R\epsilon\circ f&=&\phi_0\circ \epsilon'. \eqlabel{f2}
\end{eqnarray}
\end{lemma}

\begin{proof}
Given a 
morphism $g\in\Hom_{\Cc}(LD,C)$, we compute the corresponding element in $\Hom_\Dd(D,RC)$ as $f=Rg\circ \alpha_D$, where $\alpha$ is the unit of the adjunction. Conversely, given $f\in \Hom_\Dd(D,RC)$, we compute $g = \beta_C\circ Lf: LD\to C$, where $\beta:LR\to 1$ is the counit of the adjunction.
Suppose $f$ satisfies \equref{f1}, then 
\begin{eqnarray*}
\Delta\circ g&=&\Delta\circ \beta_C\circ Lf=\beta_{C\ot C}\circ L R\Delta\circ Lf\\
&=& \beta_{C\ot C}\circ L\phi_{C,C}\circ  L(f\odot f)\circ L\Delta'\\
&=&(\beta_{C}\ot\beta_{C})\circ \psi_{RC, R C}\circ L(f\odot f)\circ L\Delta'\\
&=&(\beta_{C}\ot\beta_{C})\circ ( Lf\ot  Lf)\circ\psi_{D, D}\circ  L\Delta'\\
&=& (g\ot g)\circ \Delta_{\uL \ol D}\\
\end{eqnarray*}
where we used naturality of $\beta$ in the second equality, \equref{f1} in the third equality, the last diagram of \thref{monoidal&adjoint} in the fourth equality, the naturality of $\psi$ in the fifth equality and the construction of the functor $\uL$ in the last equality. Hence $g$ preserves the comultiplication.
In a similar way, one proves that if $f$ satisfies \equref{f2}, then $g$ preserves the counit.

Conversely, by similar computations, we obtain that if $g$ is a coalgebra morphism, then $f$ satisfies \equref{f1} and \equref{f2}.
\end{proof}

We also state explicitly the dual of \leref{dualcoalgebramorphism}.

\begin{lemma}\lelabel{dualalgebramorphism}
Let $(L,R)$ be a pair of adjoint functors between monoidal categories $\Cc$ and $\Dd$ such that $R$ is a monoidal functor (equivalently, $L$ is op-monoidal). Consider an algebra $(A,m,u)$ in $\Dd$ and an algebra $(B,m',u')$ in $\Cc$. Then the natural isomorphism
$$\theta_{A,B}:\Hom_\Cc(LA,B)\to \Hom_\Dd(A,RB)$$
induces a natural isomorphism between $\Hom_{\Alg(\Dd)}(A,RB)$ and the set of all $f\in\Hom_\Cc(LA,B)$ that satisfy
\begin{eqnarray}
f\circ Lm &=& m'\circ (f\ot f)\circ \psi_{B,B} \eqlabel{f3} \\
f\circ Lu &=& u'\circ \psi_0 \eqlabel{f4}
\end{eqnarray}
\end{lemma}

\begin{theorem}\thlabel{uoadjunction}
Let $(L,R)$ be a pair of adjoint functors between braided monoidal categories $\Cc$ and $\Dd$. Suppose that the pair $(L,R)$ is liftable and use notations as above. Then the pair of induced functors 
$$\xymatrix{\Bialg(\Cc) \ar@<.5ex>[rr]^-{\uoR} && \Bialg(\Dd) \ar@<.5ex>[ll]^-{\uoL} }$$
is an adjoint pair of functors.

I.e. for any pair of bialgebras $H\in \Bialg(\Cc)$ and $H'\in \Bialg(\Dd)$ there is a natural isomorphism
\begin{equation}\eqlabel{uoadjunction}
\Hom_{\Bialg(\Cc)}(\uoL H',H)\cong \Hom_{\Bialg(\Dd)}(H',\uoR H).
\end{equation}
Consequently, the functors $(\uoL,\uoR)$ induce also an adjoint pair between the categories of Hopf algebras in $\Cc$ and $\Dd$ \textcolor{blue}{provided $\kappa_{\ol H'}$ is an epimorphism in $\Cc$ and $\kappa'_{\ul H}$ is an epimorphism in $\Dd$ whenever $\uoH$ is a Hopf algebra in $\Cc$ and $\uoH'$ is a Hopf algebra in $\Dd$}.
\end{theorem}

\begin{proof} 
Consider $f\in \Hom_{\Bialg(\Cc)}(\uoL H',H)$. This means $f\in\Hom_{\Alg(\Cc)}(\ol LH', H)$ such that $f$ is moreover a coalgebra morphism. By \leref{dualcoalgebramorphism}, applied to the adjoint pair $(\oL,\oR)$, the set of these morphisms $f$ is naturally isomorphic with the set of $f'\in\Hom_{\Alg(\Dd)}(H',\oR H)$ that satisfy
\begin{eqnarray*}
\ol R\Delta \circ f'&=&\ol \phi_{H,H}\circ (f'\odot f')\circ \Delta' \eqlabel{f'1}\\
\ol R\epsilon\circ f'&=&\ol \phi_0\circ \epsilon'. \eqlabel{f'2}
\end{eqnarray*}
Moreover, as $\ol R$ commutes with the forgetful functors to $\Cc$ and $\Dd$, we find that the set of morphisms $f'$ as above is isomorphic with the set of morphisms 
$g\in\Hom_\Dd(H',RH)$ that satisfy the following four conditions
\begin{eqnarray*}
R\Delta \circ g&=& \phi_{H,H}\circ (g\odot g)\circ \Delta' \eqlabel{g1}\\
R\epsilon\circ g&=& \phi_0\circ \epsilon'. \eqlabel{g2}\\
g\circ m'&=&Rm\circ \phi_{H,H}\circ (g\odot g), \eqlabel{g3}\\
g\circ u'&=&Ru\circ \phi_0. \eqlabel{g4}
\end{eqnarray*}
Using \leref{dualcoalgebramorphism} and \leref{dualalgebramorphism} applied on the pair $(L,R)$, we find that there is a bijective correspondence between the morphisms $g$ as above and morphisms $g'\in\Hom_\Cc(LH',H)$ satisfying the following four conditions
\begin{eqnarray}
\Delta\circ g'&=& (g'\ot g')\circ \psi_{H',H'}\circ L\Delta' =(g'\ot g')\circ \Delta_{LH'} \eqlabel{g'1} \\
\epsilon\circ g'&=& \psi_0\circ L\epsilon_0 = \epsilon_{LH'} \eqlabel{g'2} \\
g'\circ Lm &=& m'\circ (g'\ot g')\circ \psi_{H',H'} \eqlabel{g'3} \nonumber\\
u\circ \psi_0 &=& g'\circ Lu' \eqlabel{g'4}\nonumber 
\end{eqnarray}
As we can lift $L$ to the functor $\uL$, and equations \equref{g'1} and \equref{g'2} mean that that $g'$ is a coalgebra morphism, we find that $g'$ is in bijective correspondence with a lifted morphism $h\in\Hom_{\CoAlg(\Cc)(\ul L H',H)}$
 satisfying 
\begin{eqnarray*}
h\circ \uL m &=& m'\circ (h\ot h)\circ \ul\psi_{H',H'} \eqlabel{h3} \\
u\circ \ul\psi_0 &=& h\circ \uL u' \eqlabel{h4}
\end{eqnarray*} 
So if we apply one last time \leref{dualalgebramorphism} to the adjunction $(\uL,\uR)$, we find that the morphism $h$ above corresponds canonically with a morphism $h'\in\Hom_{\CoAlg(\Dd)}(H',\uR H)$ that is moreover an algebra morphism, i.e. $h'\in\Hom_{\Bialg(\Dd)}(H',\uoR H)$. This proofs the adjunction between the categories of bialgebras. The theorem for Hopf algebras follows since morphisms of Hopf algebras are morphisms between the underlying bialgebras.
\end{proof}

In the situation of \thref{uoadjunction}, the pair of adjoint functors $(\uoL,\uoR)$ is completing the following diagram
\[
\xymatrix{
\Bialg(\Cc) \ar[dr] \ar@<.5ex>[rr]^-{\uoR} \ar[dd] && \Bialg(\Dd) \ar@<.5ex>[ll]^-{\uoL} \ar'[d][dd] \ar[dr]\\
&\Alg(\Cc) \ar@<.5ex>[rr]^(0.3){\oR} \ar[dd]%^(0.3){U_\Cc} 
&& \Alg(\Dd) \ar@<.5ex>@{.>}[ll]^(0.7){\oL} \ar[dd]%^{U_\Dd} 
\\
\CoAlg(\Cc) \ar@<.5ex>@{.>}'[r]^-{\uR}[rr] \ar[dr] && \CoAlg(\Dd)  
\ar@<.5ex>'[l][ll]^-{\uL} \ar[dr] \\
&\Cc \ar@<.5ex>[rr]^R && \Dd \ar@<.5ex>[ll]^L
}
\]
where all undecorated vertical and lateral arrows are forgetful functors. Moreover, we can replace $\Bialg$ by $\Hpfalg$ in the above diagram, \textcolor{blue}{given that the extra necessary conditions stated in \thref{uoadjunction} are fulfilled.} {\color{blue} As we will show below, these conditions is always fulfilled in our cases of interest.}

\subsection{Lifting of an adjunction to Lie algebras}

Suppose now that $\Cc$ and $\Dd$ are additive, symmetric monoidal categories, and that $R:\Cc\to\Dd$ is an additive symmetric monoidal functor with left adjoint $L$. As known (see e.g. \cite{GV}) an additive symmetric monoidal functor preserves Lie algebras. Hence, we obtain a new functor
$$\xymatrix{ 
\LieAlg(\Cc)\ar[rr]^\wR && \LieAlg(\Dd)
}$$
which commutes with the forgetful functors. Here, $\LieAlg(\Cc)$ denotes the category of Lie algebras in $\Cc$ (the morphisms being the ones in $\Cc$ that preserve Lie brackets).
Now, recall that we have a functor 
$$\Ll_\Cc:\Alg(\Cc)\to \LieAlg(\Cc),$$
which takes any algebra $(A,m,u)$ in $\Cc$ to its commutator Lie algebra $\Ll_\Cc(A)=(A,\lie)$, where the commutator bracket $\lie$ is defined as
$$\lie:\xymatrix{A\ot A \ar[rrr]^-{m - m\circ\gamma_{A,A}} &&& A}.
$$
By construction, we have 
\begin{equation}\eqlabel{commutatorlift}
\Ll_\Dd \oR = \wR \Ll_\Cc.
\end{equation}
In the previous Section, we called the pair $(L,R)$ liftable if the adjunction could be lifted to adjunctions between the categories of algebras and coalgebras.
Similarly, we will say that $L$ is {\em Lie-liftable} if there exists a functor $\wL:\LieAlg(\Dd)\to\LieAlg(\Cc)$
that is a left adjoint for $\wR$.
As we will remark in the application section, this condition is not too hard, as $\wL$ can usually be constructed out of $\oL$, provided this functor exists. We then obtain the following diagram of functors
\[
\xymatrix{
\LieAlg(\Cc)\ar@<.5ex>[rr]^\wR && \LieAlg(\Dd)\ar@{.>}@<.5ex>[ll]^\wL\\
\Alg(\Cc) \ar@<.5ex>[rr]^\oR \ar[u]^{\Ll_\Cc} && \Alg(\Dd)\ar@<.5ex>[ll]^\oL \ar[u]^{\Ll_\Dd}
}
\]

Moreover, it is folklore knowledge that the construction of the universal envelope of a Lie algebra generalizes to cocomplete abelian monoidal categories where endofunctors of the form $-\ot X$ and $X\ot -$ are exact. Because we want to avoid too many restrictions, we will just suppose on $\Cc$ the existence of an {\em enveloping} functor 
$$U:\LieAlg(\Cc)\to \Alg(\Cc)$$
that is a left adjoint to the functor $\Ll:\Alg(\Cc)\to \LieAlg(\Cc)$. We will denote the unit and counit of this adjunction respectively by $\nu$ and $\theta$.
As in the classical case, the enveloping algebra of a Lie algebra is a Hopf algebra, as can be seen from the proposition presented here below; as we believe this is folklore, we only give a sketch of the proof. Let us first recall the notion of {\em primitive elements} $P(H)$ of a Hopf algebra $H$ in $\Cc$. Remark that $\Delta_H$ and $\ol u_H:=u\ot H+H\ot u$ are Lie algebra morphisms $\Ll H\to \Ll(H\ot H)$. We then define $P(H)$ as the object part in the following equalizer in $\LieAlg(\Cc)$.
\[
\xymatrix{P(H)\ar[rr]^p && \Ll H \ar@<.5ex>[rr]^-{\Delta_H} \ar@<-.5ex>[rr]_-{\ol u_H} && \Ll(H\ot H)}
\]
Remark that if $\Cc$ has equalizers and endofunctors of the form $-\ot X$ preserve equalizers, then $P(H)$ can be computed as an equalizer in $\Cc$, endowed with the restriction of the Lie bracket of $\Ll H$.

\begin{proposition}\prlabel{adjunctionUP}
If $\Cc$ is an additive symmetric monoidal category with equalizers, that admits a pair $(U,\Ll)$ of adjoint functors  as above, then these functors restrict and corestrict to a pair $(U,P)$ of adjoint functors 
\[
\xymatrix{\LieAlg(\Cc) \ar@<.5ex>[rr]^-U && \Hpfalg(\Cc) \ar@<.5ex>[ll]^-P}
\]
%we $P(H)$ is called the Lie algebra of primitive elements in $H$.
\end{proposition}
\begin{proof}
First, let us prove that for any Lie algebra $K$ in $\Cc$, $U(K)$ is a Hopf algebra in $\Cc$. To prove that $U(K)$ is a bialgebra, we must show that $U(K)$ is a coalgebra in $\Alg(\Cc)$. By the adjunction $(U,\Ll)$, the existence of a comultiplication $\Delta_{UK}\in\Hom_{\Alg(\Cc)}(UK,UK\ot UK)$ is equivalent with the existence of a map $\ol\Delta\in\Hom_{\LieAlg(\Cc)}(K,\Ll(UK\ot UK))$. Since $\Ll(UK\ot UK)$ is just the $\Cc$-object $UK\ot UK$ endowed with the commutator Lie algebra structure, we can define $\ol\Delta=\nu_K\ot u_{UK}+ u_{UK}\ot\nu_K$, which is easily checked to be a Lie algebra morphism. Similarly, to define a counit $\epsilon_{UK}\in\Hom_{\Alg(\Cc)}(UK,I)$ is equivalent to finding a morphism $\ol\epsilon\in\Hom_{\LieAlg(\Cc)}(K,\Ll I)$. Remark that the Lie bracket on $\Ll I$ is just the zero map, hence we take $\ol\epsilon=0$. With these structure maps, $UK$ becomes a (cocommutative) bialgebra in $\Cc$. Finally, for $UK$ to be a Hopf algebra we need to define an antipode $S\in\Hom_{\Alg(\Cc)}(UK,UK^{op})$, which is equivalent with finding a morphism $\ol S\in \Hom_{\LieAlg(\Cc)}(K,\Ll(UK^{op}))$. One checks that taking $\ol S=-\nu_K$ leads to the correct structure. 
%, hence $U$ can be corestricted to a functor 
%$$U:\LieAlg(\Cc)\to \Hpfalg(\Cc).$$

The other way around, if $H$ is a Hopf algebra, we define the Lie algebra $P(H)$ to be the primitive elements of $H$, endowed with the restriction of the commutator Lie algebra of $H$.

Let us construct the unit $\ol\nu$ of the adjunction $(U,P)$. Given a Lie algebra $(K,\lie)$, consider the unit $\nu_K:(K,\lie)\to \Ll U(K,\lie)$ of the adjunction $(U,\Ll)$.
From the discussion above, it follows that $\Delta_{UK}\circ \nu_K=\ol\Delta$. 
Furthermore, we also have $\ol u_{UK}\circ \nu_K=\ol\Delta$. Hence, $\nu_K:K\to \Ll UK$ 
equalizes the following morphisms
$$\xymatrix{
\Ll UK \ar@<.5ex>[rr]^-{\Delta_{UK}} \ar@<-.5ex>[rr]_-{\ol u_{UK}} && \Ll(UK\ot UK)
}$$
and consequently, we find by the universal property of the equalizer $(PUK,p)$ a unique Lie-algebra morphism $\ol\nu_K:K\to PUK$ such that $\nu_K=p\circ\ol\nu_K$. 

In a similar way, one constructs the counit $\ol\theta$ out of the counit $\theta$ of the adjunction $(U,\Ll)$. The remaining details are left to the reader.
\end{proof}

Recall that if $H$ is a Hopf algebra in $\Cc$, then $\oR H$ is an algebra in $\Dd$, but not necessarily a Hopf algebra in $\Dd$. As $\oR H$ is not a Hopf algebra, it makes no sense to speak about the primitive elements of $\oR H$. However, we can still consider the following equalizer in $\LieAlg(\Dd)$, of which the object part is referred to as the {\em pre-primitive elements} of $\oR H$.
\[
\xymatrix{
P'\ol R H \ar[rr]^-{p'} && \Ll\ol R H \ar@<.5ex>[rr]^-{\ol R\Delta_H} \ar@<-.5ex>[rr]_-{\ol\phi_{H,H}\circ \ol u_{\ol R H}} && \Ll\ol R (H\ot H)
}
\]
As for usual primitives, to compute this equalizer one needs to remark first that the morphisms $\ol R\Delta_H$ and $\ol\phi_{H,H}\circ \ol u_{\ol R H}$ are Lie algebra morphisms.
The following proposition explains how pre-primitive elements can be given an easy interpretation: they coincide exactly with the primitives of $H$, on which the functor $\wR$ is applied.
\begin{proposition}
With notation and conventions as above, for any Hopf algebra $H$ in $\Cc$, there is a natural isomorphism of Lie algebras in $\Dd$
$$P'\oR H\simeq \wR PH$$
\end{proposition}
\begin{proof}
Every right adjoint functor preserves limits. Hence, since $PH$ is the equalizer of $(\Delta_H,\ol u_H)$, $\wR P H$ will still be the equalizer of $(\oR\Delta_H,\oR\ol u_H)$. Furthermore, by the naturality of $\ol\phi$ and the construction of $\ol u_{\oR H}$ (in particular the construction of $u_{\oR H}$, see \thref{liftingalgebras}), we find easily that that $\oR\ol u_H=\ol\phi_{H,H}\circ \ol u_{\oR H}$, so $\wR P H$ and $P' \oR H$ are naturally isomorphic by uniqueness of the equalizer.
\end{proof}

\subsection{Michaelis' theorem in monoidal categories}

With the observations of the previous Section in mind, suppose that the initial pair of adjoint functors $(L,R)$ is both liftable and Lie-liftable% \textcolor{blue}{and that the conditions of \thref{uoadjunction} are fulfilled}
, then we can construct a diagram of functors as follows.
\[
\xymatrix{
{\color{blue}\Bialg}(\Cc) \ar@<.5ex>[rr]^-{\uoR} \ar@<.5ex>[d]^{P} && {\color{blue}\Bialg}(\Dd) \ar@<.5ex>[d]^{P} \ar@<.5ex>[ll]^-{\uoL} \\
\LieAlg(\Cc) \ar@<.5ex>[u]^{U} \ar@<.5ex>[rr]^{\wR} && \LieAlg(\Dd) \ar@<.5ex>[u]^{U} \ar@<.5ex>[ll]^-\wL
}
\]
Our aim is now to understand to which extent this diagram commutes. Firstly, we will prove that for any liftable and Lie-liftable pair $(L,R)$% \textcolor{blue}{satisfying the conditions of \thref{uoadjunction}}
, and for any Lie algebra $K$ in $\Dd$, we have the natural isomorphism of {\color{blue} bi}algebras $\uoL UK\cong U\wL K$ in $\Cc$, {\color{blue} from which it follows in particular that $\uoL UK$ is always a Hopf algebra, even if the additional conditions of \thref{uoadjunction} are not fulfilled, since we know that $U\wL K$ is a Hopf algebra and the category of Hopf algebras is a full and replete subcategory of the category of bialgebras.} 
We will show this in several steps.

\begin{lemma}\lelabel{uoLwL}
Let $(L,R)$ be a pair of liftable and Lie-liftable adjoint functors between symmetric monoidal categories. For a Lie algebra $K$ in $\Dd$, there is a natural isomorphism of algebras in $\Cc$
$$U\wL K\cong \uoL UK$$
\end{lemma}

\begin{proof}
We use the fact that, by definition, $U$ is the left adjoint of $\Ll$. Therefore the isomorphism will follow if we find for any algebra $A$ in $\Cc$ that 
$$\Hom_{\Alg(\Cc)}(\uoL UK, A)\cong \Hom_{\LieAlg(\Cc)}(\wL K,\Ll A).$$
We compute
\begin{eqnarray*}
\Hom_{\Alg(\Cc)}(\uoL UK, A)&=& \Hom_{\Alg(\Cc)}(\oL UK, A)\\
&\cong &\Hom_{\Alg(\Dd)}(UK, \oR A)\\
&\cong &\Hom_{\LieAlg(\Dd)}(K,\Ll\oR A)\\
&\cong &\Hom_{\LieAlg(\Dd)}(K,\wR\Ll A)\\
&\cong &\Hom_{\LieAlg(\Cc)}(\wL K, \Ll A)
\end{eqnarray*}
\end{proof}

\begin{theorem}\thlabel{uoLwL}
Let $(L,R)$ be a pair of liftable and Lie-liftable adjoint functors between symmetric monoidal categories. 
%\textcolor{blue}{such that the conditions of \thref{uoadjunction} are fulfilled}. 
Then for any Lie algebra $K$ in $\Dd$, the natural isomorphism $\omega:U\wL K\to \uoL UK$ from \leref{uoLwL} is a Hopf algebra isomorphism. {\color{blue} In particular, $\uoL UK$ is a Hopf algebra.}
%there is a natural isomorphism of Hopf algebras in $\Cc$
%$$\uoL UK\cong U\wL K$$
\end{theorem}

\begin{proof}
From the proof of \leref{uoLwL}, we know that the isomorphism $\omega$ corresponds to a morphism $\ul\omega:\wL K\to \Ll\uoL UK$ that is given explicitly by
\[
\xymatrix{
\wL K \ar[rrr]^-{\ul\omega} \ar[d]_{\wL\nu_K} &&& \Ll\oL UK \\
\wL \Ll UK \ar[rr]^-{\wL\Ll\ol\alpha_{UK}} && \wL\Ll \oR\oL UK \ar@{=}[r] & \wL\wR\Ll\oL UK \ar[u]_-{\widetilde\beta_{\Ll\oL UK}} %&& \Ll\oL UK \ar[u]_-{\theta_{\oL UK}}
}
\]
Remark (see \leref{functorsuo}) that $\oL UK$ can be endowed with the structure of a %Hopf 
{\color{blue}bi}algebra -that we denoted by $\uoL UK$- whose comultiplication is given by $\Delta_{\uoL UK}=\ol\psi_{UK,UK}\circ \oL\Delta_{UK}$. Since $\Hom_{{\color{blue}\Bialg}}(U\wL K,\uoL UK)\cong\Hom_{\LieAlg}(\wL K,P\uoL UK)$, in order to obtain that $\omega$ is a {\color{blue}bi}algebra morphism, we need to prove that the image of $\ul\omega$ lies in $P\uoL UK$, or more precisely that $\ul\omega=p\circ\ul{\ol\omega}$ for some morphism $\ul{\ol\omega}:\wL K\to P\uoL UK$. From \prref{adjunctionUP}, we know that $\nu_K$ splits as $p\circ \ol\nu_K$. Now we can compute 
\begin{eqnarray*}
\oR\Delta_{\uoL UK}\circ \ol\alpha_{UK}\circ p&=&
\oR\ol\psi_{UK,UK}\circ \ol R\oL\Delta_{UK} \circ \ol\alpha_{UK}\circ p\\
&=&\oR\ol\psi_{UK,UK}\circ \ol\alpha_{(UK\ot UK)}\circ \Delta_{UK}\circ p\\
&=& \ol\phi_{\oL UK,\oL UK}\circ (\ol\alpha_{UK}\ot \ol \alpha_{UK})\circ \Delta_{UK}\circ p\\
&=& \ol\phi_{\oL UK,\oL UK}\circ (\ol\alpha_{UK}\ot \ol \alpha_{UK})\circ \ol u_{UK} \circ p\\
&=& \ol\phi_{\oL UK,\oL UK}\circ \ol u_{\oR\oL UK}\circ \ol\alpha_{UK} \circ p\\
\end{eqnarray*}
Here we used the naturality of $\ol\alpha$ in the second and in the last equality, \thref{monoidal&adjoint} in the third equality and the equalizing property of $p$ in the fourth equality.
It follows from the above computations that the map $\ol\alpha_{UK}\circ p:PUK\to \oR\uoL UK$ equalizes the pair $(\oR\Delta_{\uoL UK},\ol\phi_{\oL UK,\oL UK}\circ \ol u_{\oR\oL UK})$. Since the equalizer of this pair is given exactly by the pre-primitives $(P'\oR\uoL UK,p')$, we obtain by the universal property of the equalizer a unique morphism $\ol{\ol\alpha}:PUK\to P'\oR\oL UK$ with the property that $\ol\alpha_{UK}\circ p=p'\circ \ol{\ol \alpha}$. Therefore, we find that $\ul\omega=p\circ\ol{\ul\omega}$ as in the following commutative diagram
\[
\xymatrix{
\wL K \ar[rrr]^-{\ul{\ol\omega}} \ar[d]_{\wL\ol\nu_K} &&& P\uoL UK \ar[rr]^-{p} && \Ll\uoL UK=\Ll\oL UK\\
\wL P UK \ar[d]_{\wL p} \ar[rr]^-{\wL\ol{\ol\alpha}_{UK}} && \wL P' \oR\oL UK \ar[d]_{\wL p'} \ar@{=}[r] & \wL\wR P\uoL UK \ar[u]_-{\widetilde\beta_{P\uoL UK}} \ar[d]^{\wL\wR p} \\
\wL \Ll UK  \ar[rr]^-{\wL\Ll\ol\alpha_{UK}} && \wL\Ll \oR\oL UK \ar@{=}[r] & \wL\wR\Ll\oL UK \ar[rruu]_-{\widetilde\beta_{\Ll\oL UK}}
}
\]
In conclusion, by considering the corestriction $\ol{\ul\omega}$ of $\ol\omega$, we find that the isomorphism $\omega:U\wL K\to \uoL UK$ is a {\color{blue}bi}algebra isomorphism. {\color{blue} As $U\wL K$ is moreover a Hopf algebra and the $\Hpfalg$ is a full and replete subcategory of $\Bialg$, it follows that $\uoL UK$ is also a Hopf algebra and $\omega$ is a Hopf algebra isomorphism.}
\end{proof}

\begin{theorem}\thlabel{Michaelis}
Let $(L,R)$ be a pair of liftable and Lie-liftable adjoint functors between symmetric monoidal categories% \textcolor{blue}{such that the conditions of \thref{uoadjunction} are fulfilled}
.
Let $H$ be a Hopf algebra in $\Cc$, then there is an isomorphism of Lie algebras in $\Dd$
$$P\uoR H\cong \wR P H$$
\end{theorem}

\begin{proof}
By composition of adjunctions, we know that $P\uoR$ is a right adjoint of $\uoL U$ and $\wR P$ is a right adjoint of $U\wL$. As there is a natural isomorphism $\uoL U\simeq U\wL$ (see \thref{uoLwL}), the statement follows by the uniqueness of the adjoint.
\end{proof}

\section{The classical result of Michaelis}\selabel{Mich}

In this Section, we show that our theorem is a proper generalization of Michaelis' original result.

Consider the abelian symmetric monoidal category of vector spaces $\Dd=(\Vect,\ot,k)$, the symmetry being the twist. We put $\Cc$ to be the dual category of $\Dd$, where both arrows and the tensor product are conversed, that is $\Cc=(\Vect^{op},\ot^{op},k)$, i.e.
$$\Hom_\Cc(X,Y)=\Hom_\Dd(Y,X), \qquad X\ot^{op} Y= Y\ot X,$$
for any two vector spaces $X,Y$. Taking the vector space dual $X^*=\Hom_k(X,k)$ then provides a (covariant) adjunction
$$
\xymatrix{
\Cc \ar@<.5ex>[rr]^-{R=(-)^*} && \Dd \ar@<.5ex>[ll]^-{L=(-)^*}.
}
$$
Remark that if $R$ and $L$ are considered as endofunctors on $\Vect$, this is a contravariant adjunction. This is expressed by the following natural isomorphism
$$\Hom_{\Vect^{op}}(Y^*,X)=\Hom_{\Vect}(X,Y^*)\cong \Hom_{\Vect}(Y,X^*)\quad (\cong \Hom_{\Vect}(X\ot Y,k)),$$
for $X,Y\in \Vect$ (but considered as $X\in\Cc$ and $Y\in \Dd$). Both units are given by the same linear map, known as the canonical injection of a vector space into its bidual space
$$\iota_X:X\to X^{**}.$$

On the one hand, algebras, coalgebras, bialgebras, Hopf algebras, Lie algebras and Lie coalgebras in $\Dd=\Vect$ clearly are the classical ones. On the other hand, algebras, coalgebras, bialgebras, Hopf algebras, Lie algebras and Lie coalgebras in the opposite category $\Cc$ are, respectively, classical coalgebras, algebras, bialgebras, Hopf algebras, Lie coalgebras and Lie algebras. Remark that, as classically known, bialgebras and Hopf algebras are the only self-dual concepts among these. 

Furthermore, we know that for any two vector spaces $X$ and $Y$, there is a natural map
$$X^*\ot Y^*\to (Y\ot X)^*.$$
This means exactly that the functor $(-)^*$, considered as the covariant functor $R$, is monoidal, or equivalently, that the covariant functor $L$ is op-monoidal. By applying \thref{liftingalgebras} we find that $R$ induces a functor 
$\ol R:\Alg(\Cc)\to \Alg(\Dd)$. Explicitly, we obtain in our situation the well-known functor
$$(-)^*:\CoAlg_k\to \Alg_k$$
that computes the dual algebra of a coalgebra. Remark that the functor $\uL:\CoAlg(\Dd)\to\CoAlg(\Cc)$ is in fact exactly the same functor. A classical result in coalgebra theory tells us that a left adjoint for the functor $\oR$ is given by the functor that assigns the so-called finite dual coalgebra $A^\circ$ to a $k$-algebra $A$, see e.g. \cite[Theorem 2.3.14]{Abe}. Remark also that, by contravariance of the initial functor $(-)^*$, the same construction gives a left adjoint for the functor $\oR$.
Hence the pair of adjoint functors $(L,R)$ induced by taking vector space duals is a liftable pair (in the terminology of \reref{liftable}) and we can apply the results obtained in \seref{liftings}. 
{\color{blue}Remark that in this case, the natural transformation $\kappa: L\circ U_\Dd\to U_\Cc \circ \oL$ from \equref{kappa}, is given exactly by the embedding of the Sweeder dual into the linear dual, which is of course an epimorphism in $\Vect^{op}$. Hence the additional conditions of \thref{uoadjunction} are automatically fulfilled.}
%, \textcolor{blue}{as $A^{\circ}$ is always nicely embedded in $A^{*}$}. 
In particular, we recover the result that the finite dual induces an auto-adjunction on the category of Hopf algebras (see \cite{Abe}, page 87).

Finally, if we consider classical $k$-Lie algebras, we find moreover that $(-)^*$ is also Lie-liftable. Indeed, recall from \cite{Mich} that given $C$ is a Lie coalgebra, there is a natural Lie algebra structure on the dual vector space $C^*$ of $C$. Hence we find a functor $\LieCoAlg_k\to \LieAlg_k$ or, in our notation from above, a functor $\wR:\LieAlg(\Cc)\to \LieAlg(\Dd)$, and this functor admits a left adjoint that computes for any $k$-Lie algebra $K$ the Lie coalgebra on the space $K^\circ$ of all functionals $f\in K^*$ that satisfy 
$$f^*([x,y])=\sum_if_i(x)g_i(y)$$
for certain $f_i,g_i\in K^*$. We then obtain from \thref{uoLwL} that 
$$U^c(K^\circ)\cong (UK)^\circ$$
where $U^c$ denotes the universal enveloping Hopf algebra of a Lie coalgebra. This result was proven in \cite[page 31-32]{Mich}. 
Our result \thref{Michaelis} comes out as 
$$P(H^{\circ})\cong Q(H)^*$$
which is the original result of Michaelis \cite{Mich2}, where $Q(H)$ denotes the Lie coalgebra of indecomposables of a Hopf algebra. More precisely, $Q(H)=I/I^2$, where $I=\ker\varepsilon$.

\begin{remark}\relabel{gegradeerdeversieMichaelis}
Putting $\Dd=(\Vect^{\ZZ_{2}},\ot,k)$ and $\Cc=({\Vect^{\ZZ_{2}}}^{op},\ot^{op},k)$ ($\Vect^{\ZZ_{2}}$ being the category of $\ZZ_{2}$-graded vector spaces and degree-preserving morphisms), we now consider the symmetry in both categories to be the following one: for any pair of objects $(V,W)$ in $\Vect^{\ZZ_2}$, we put $$c_{V,W}:V\ot W\to W\ot V; v\ot w\mapsto (-1)^{|v||w|}w\ot v.$$
One then computes easily that algebras in $\Dd$ are precisely superalgebras, Lie algebras in $\Dd$ are exactly Lie superalgebras, etc. Now, replacing $X^{*}$ by $\HOM_{k}(X,k)$ ($\HOM_{k}(-,k)$ being the right adjoint to the endofunctor $-\ot k$ on $\Vect^{\ZZ_{2}}$) in the above, we get a ``super'' version of Michaelis' theorem.
\\Now recall the Hom-construction from \cite{CG}. Using same notation as in the cited paper, put $\Dd=\widetilde{\Hh}(\Vect)$ and $\Cc=\widetilde{\Hh}(\Vect)^{op}$. Propositions 1.1 and 1.2 in \cite{CG} assert that $\Dd$ and $\Cc$ are both additive, symmetric (non-strict) monoidal categories. Taking also Proposition 1.6 into account, one can develop a ``Hom'' version of Michaelis' theorem, along the same lines as here above.
\end{remark}
\begin{remark}\relabel{Prost}
\textcolor{blue}{Instead of considering vector spaces over a field% (of characteristic zero)
, let us take the category of modules $\Mm_{R}$ over an arbitrary commutative ring $R$. Just as in the example elaborated above, taking duals determines a liftable pair of functors also here, so we still get a lifted adjunction at the bialgebra level.\\
However, for the lifted adjunction to map $R$-Hopf algebras to $R$-Hopf algebras, extra conditions are necessary in order for the conditions of \thref{uoadjunction} to be fulfilled. \cite[Theorem 51]{PS} asserts that when $R$ is noetherian and absolutely flat, these conditions are met.} {\color{blue}Remark however, that the theorem of Michaelis still holds in this setting, as it follows from our general theory (see \thref{uoLwL}) that even in this case $\uoL UK$ is still a Hopf algebra for any Lie algebra $K$.}
\end{remark}
\section{Applications to Turaev's Hopf group-coalgebras}\selabel{Tur}

\subsection{Dualities for Hopf group-coalgebras}\selabel{Fam}

A classical construction in category theory is the so-called category of families over a given category $\Aa$, which we will denote as $\Fam(\Aa)$. Recall that objects in this category are pairs $(I,\{X_i\}_{i\in I})$, where $I$ is any (index) set and $\{X_i\}_{i\in I}$ are objects in $\Aa$, indexed by the set $I$. We will write briefly $\ul X=(I,X_i)$. We call the vector space $X_i$ the {\em homogeneous part of degree} $i$ of the object $\ul X$. \index{part!homogeneous}
 A morphism $\ul\phi:\ul X\to \ul Y=(J,Y_j)$ is a pair $\ul\phi=(f,\phi_i)$, where $f:I\to J$ is just a map and $\phi_i:X_i\to Y_{f(i)}$ is a family of morphisms in $\Aa$, indexed by the set $I$. Furthermore, if $\Aa$ is a monoidal category with product $\ot$ and unit $k$, then $\Fam(\Aa)$ is a monoidal category as well, with product  given by
\begin{eqnarray*}
\ul X\ul \ot \ul Y&=&(I\times J,\{X_i\ot Y_j\}_{(i,j)\in I\times J})\\
\ul \phi\ul\ot\ul\psi&=&(f\times g, \{\phi_i\ot\psi_j\}_{(i,j)\in I\times J})
\end{eqnarray*}
for objects $\ul X=(I,X_i)$ and $\ul Y=(J,Y_j)$ in $\Fam(\Aa)$ and for morphisms $\ul\phi=(f,\phi_i):\ul X\to \ul X'$ and $\ul\psi=(g,\psi_j):\ul Y\to \ul Y'$.
The unit is $\ul k=(\{*\},k)$, where $\{*\}$ is a fixed singleton set. Moreover, if $\Aa$ is braided or symmetric with braiding $c_{-,-}$, then $\Fam(\Aa)$ is braided or symmetric as well, with braiding 
$$c_{\ul X,\ul Y}=(\tau_{I,J},c_{X_i,Y_j}),$$
where $\tau$ denotes the usual twist in $\Set$. 

We denote $\Fam(\Aa^{op})^{op}=\Maf(\Aa)$. Remark that the objects in $\Fam(\Aa)$ and $\Maf(\Aa)$ are the same, but morphisms in $\Maf(\Aa)$ are pairs $(f,\phi_j):\ul X\to \ul Y$, where $f:J\to I$ and $\phi_j:X_{f(j)}\to Y_{j}$.

In \cite{CaeDel} algebras, coalgebras, bialgebras and Hopf algebras in $\Fam(\Aa)$ and $\Maf(\Aa)$ were computed. 
We will briefly repeat this here again, pointing out some interesting observations that we believe were omitted in the paper cited above, but we refer the reader for more detailed computations to that paper. 

\subsubsection{Algebras, coalgebras and Hopf algebras in $\Fam(\Aa)$}
First observe that we have obvious forgetful strong monoidal functors $\Fam(\Aa)\to \Set$ and $\Maf(\Aa)\to \Set^{op}$. Hence, if $(I,A_i)$ is an algebra with multiplication $(m,\mu_{i,j})$ 
and unit $(e,\eta)$ (resp. coalgebra with comultiplication $(d,\Delta_i)$ and counit $(p,\epsilon_i)$) in $\Fam(\Aa)$, then $I$ will be a monoid in $\Set$ with multiplication $m$ and unit $e$ (resp. just a set enriched with the trivial diagonal map $d:I\to I\times I$ and the projection $p:I\to \{*\}$). One easily checks that the axioms on morphisms in $\Fam(\Aa)$ imply that 
$$\mu_{i,j} : A_i \ot A_j \to A_{ij} \quad {\rm and} \quad \eta : k \to A_e$$ 
satisfy
\begin{eqnarray*}
\mu_{ij,\ell}\circ (\mu_{i,j} \ot A_\ell ) &=& \mu_{i,j\ell} \circ (A_i \ot \mu_{j,\ell})\\
\mu_{i,e}\circ (A_i \ot \eta) = &A_i& = \mu_{e,i} \circ ( \ot A_i )
\end{eqnarray*}
for all $i,j,\ell\in I$ (respectively all $(A_i,\Delta_i,\epsilon_i)$ are coalgebras in $\Aa$). In case $I$ is a group, the pair $(I,A_i)$ was called an $I$-algebra (group-algebra, or monoid-algebra in the general case) in \cite{Zun} (this notion should not be confused with a group algebra $kG$ !). If $\Aa$ has coproducts and if endofunctors of the form $-\ot X$ preserve coproducts (e.g. when $\Aa$ is closed monoidal), then $(I,A_i)$ is an algebra in $\Fam(\Aa)$ exactly if $\coprod_{i\in I}A_i$ is an algebra graded over the monoid $I$. 
In the coalgebra case, we just find a family of coalgebras $(A_i,\Delta_i,\epsilon_i)$ indexed by a set $I$ without any further condition. Since an arbitrary coproduct of coalgebras is again a coalgebra, we find that $\coprod_{i\in I}A_i$ is a coalgebra in $\Aa$.
Combining the above, given a bialgebra $(I,A_i)$ in $\Fam(\Aa)$, we find that $I$ is a monoid and $A=\coprod_{i\in I}A_i$ is a bialgebra whose algebra is graded by a monoid and whose homogeneous spaces are subcoalgebras and the morphisms $\mu_{i,j}$ are coalgebra morphisms. Again in \cite{Zun}, the pair $(I,A_i)$ was termed a semi-Hopf $I$-algebra (semi-Hopf group-algebra or semi-Hopf monoid-algebra in the general case) in this situation, provided $I$ is a group.
Finally, let $(I,A_i)$ be a Hopf algebra in $\Fam(\Aa)$.  Applying the forgetful functor to $\Set$, we find that $I$ must be group. Furthermore $A=\coprod_{i\in I}A_i$ is a bialgebra as before, equipped with an antipode $S:A\to A$ that decomposes on the homogeneous parts as
$$S_i:A_i\to A_{i^{-1}}$$
satisfying 
$$\mu_{i^{-1},i} \circ (S_i \ot A_i) \circ \Delta_i = \eta \circ \epsilon_i = \mu_{i,i^{-1}} \circ (A_i \ot S_i) \circ \Delta_i$$
Remark that the subspace $H_e$, where $e$ denotes the unit of the group $I$, is a Hopf algebra in the usual sense.
In \cite{Zun}, $(I,A_i)$ is called a Hopf $I$-algebra (Hopf group-algebra). 

\begin{remark}\relabel{irred}
We already remarked that any Hopf group $I$-algebra can be viewed as a (usual) Hopf algebra, that is graded as an algebra over the group $I$. We would like to remark that conversely, given any Hopf algebra $H$, one can consider the group of grouplike elements $G(H)$. Moreover, it is well-known that a cocommutative coalgebra can be decomposed into a direct sum of its irreducible components (see e.g. \cite{Mont:book}, Theorem 5.6.3). Now, letting $H_x$ be the irreducible component of a Hopf algebra $H$ containing the grouplike element $x$, the multiplication of $H$ induces a strong grading over these subspaces (i.e.\ $H_xH_y=H_{xy}$) and the antipode can be restricted to morphisms $S:H_{x}\to H_{x^{-1}}$ (cf. \cite{Mont:book}, Corollary 5.6.4). Combining these results, we find that a cocommutative Hopf algebra provides examples of Hopf group $G$-algebras (where the group $G$ is the subgroup of $G(H)$ that consists of classes of grouplike elements that are in the same irreducible component). 
Hence, Hopf group-algebras could be understood as a non-cocommutative generalization of the above, although the coalgebras are of course not supposed to be irreducible.
\end{remark}

\subsubsection{Algebras, coalgebras and Hopf algebras in $\Maf(\Aa)$}
Dually, consider the category $\Maf(\Aa)$ and let $\ul A=(I,A_i)$ be an algebra therein. As we now have a strict monoidal functor $\Maf(\Aa)\to \Set^{op}$, we obtain that $\ul A$ consists of a family of algebras $A_i$ indexed by the set $I$, that is considered as a coalgebra in $\Set$. 
If $(I,A_i)$ is a coalgebra in $\Maf(\Aa)$, then we find that $I$ is a monoid in $\Set$, and we have a family of comultiplicative morphisms
$$\Delta_{i,j}:A_{ij}\to A_i\ot A_j$$
and a counit map
$$\epsilon:A_e\to k$$
satisfying
\begin{eqnarray*}
\Delta^2_{i,j,\ell}:=(\Delta_{i,j}\ot A_\ell)\circ \Delta_{ij,\ell}&=&(A_i\ot \Delta_{j,\ell})\circ \Delta_{i,j\ell}\\
(A_i\ot \epsilon) \circ \Delta_{i,e} = &A_i& = (\epsilon\ot A_i)\circ \Delta_{e,i}
\end{eqnarray*}
If $I$ is a group, the pair $(I,A_i)$ is called an $I$-coalgebra (group-coalgebra, or monoid-coalgebra in the general case) in \cite{Tur:book} (see also unpublished \cite{Tur}).
Similarly to the dual case, if $\Aa$ has products and if endofunctors of the form $-\ot X$ preserve products, then $\prod_{i\in I}A_i$ becomes a coalgebra in $\Aa$, that is ``cograded'' over $I$. However, the condition that tensor products preserve products is rather unnatural in the infinite-dimensional case. Hence, we advocate a different approach avoiding this condition. In an abelian category, we have a canonical monomorphism $\iota:\coprod_{i\in I}A_i\to \prod_{i\in I}A_i$. We denote the projections of the product $\prod_{i\in I}A_i$ by $\pi_i$ and the projections of the product $\prod_{(j,\ell)\in I\times I} A_j\ot A_\ell$ by $\pi_{j,\ell}$. Using the universal property of the product, the morphisms $\Delta_{i,j}$ induce a unique global comultiplication $\Delta:\coprod_{i\in I}A_i\to \prod_{(j,\ell)\in I\times I} A_j\ot A_\ell$ that renders the following diagram commutative
\[
\xymatrix{
\coprod_{i\in I}A_i  \ar[rr]^-{\Delta} \ar[d]_-\iota && \prod_{(j,\ell)\in I\times I} A_j\ot A_\ell \ar[dd]^{\pi_{j,\ell}}\\
\prod_{i\in I}A_i \ar[d]^{\pi_{j\ell}} \\
A_{j\ell} \ar[rr]^{\Delta_{j,\ell}} && A_j\ot A_\ell
}
\]
One can verify that this morphism is coassociative and counital in an appropriate way. In \cite{JV:KHA} it is shown that such a coalgebra is a coalgebra in an appropriately constructed Kleisli category; therefore we call $A$ a Kleisli coalgebra. 

Suppose now that $(I,A_i)$ is a bialgebra in $\Maf(\Aa)$. Then this means that $I$ is a monoid and 
$(I,A_i)$ is a monoid-coalgebra in the sense above, such that all $A_i$ are unital algebras and the counit $\epsilon$ and the comultiplicative maps $\Delta_{i,j}$ are algebra morphisms.
In this case, 
$(I,A_i)$ was called a semi-Hopf $I$-coalgebra (semi-Hopf group-coalgebra, or semi-Hopf monoid-coalgebra in the general case) in \cite{Tur}.

If $(I,A_i)$ is moreover a Hopf algebra in $\Maf(\Aa)$, then $I$ is a group and we additionally have morphisms $S_i:A_{i^{-1}}\to A_i$ satisfying
$$\mu_i \circ (S_i \ot A_i) \circ \Delta_{i^{-1},i} = \eta_i \circ \epsilon = \mu_i \circ (A_i \ot S_i) \circ \Delta_{i,i^{-1}}$$
In this situation, $(I,A_i)$ is called 
a Hopf $I$-coalgebra (Hopf group-coalgebra) (see \cite{Tur}).

It was shown in \cite{ADVanDaele} (see also \cite{JV:KHA}) that, if $(I,A_i)$ is a Hopf group-coalgebra, then $\coprod_iA_i$ is a multiplier Hopf algebra. To see this, one should observe that if $I$ is an infinite group, $A=\coprod_{i\in I}A_i$ is a non-unital algebra, though an algebra with central local units. Consequently, the multiplier algebra of $A$ is exactly given by $\MM(A)=\prod_{i\in I}A_i$. Therefore, the induced coproduct on $A$, as defined above on the underlying Kleisli coalgebra, is a map $\Delta:A\to \MM(A\ot A)$.

The above can be summarized in the following table.\\ \\
\begin{tabular}{|l||p{6.2cm}|p{6.2cm}|}
\hline
& $\Fam(\Aa)$ & $\Maf(\Aa)$\\
\hline \hline
Algebra & 
$I$-algebra (with $I$ a monoid)
& Collection of algebras indexed by a set  \\
\hline
Coalgebra & Collection of coalgebras indexed by a set & 
$I$-coalgebra (with $I$ a monoid)\\
\hline
Bialgebra & 
Semi-Hopf monoid-algebra
& Semi Hopf monoid-coalgebra\\
\hline
Hopf algebra & 
Hopf group-algebra
& 
Hopf group-coalgebra\\
\hline
\end{tabular}

\subsubsection{Equivalence between locally finite Hopf group-algebras and locally finite Hopf group-coalgebras}

Let $\Aa$ be a braided monoidal category. In this note, we will say that $\Aa$ is a {\em pre-rigid category}\index{category!pre-rigid}, if for any object $X$, there exists an object $X^*$ and a morphism $\ev_X:X^*\ot X\to I$ with the following universal property. For any other pair $(T,t:T\ot X\to I)$ there exists a unique morphism $u:T\to X^*$ such that $\ev_X\circ (u\ot X)=t$. Clearly, the basic example is to take $\Aa=\Vect$ endowed with usual dual spaces.
We then obtain the following result.

\begin{proposition}\prlabel{prerigid}
If $\Aa$ is a pre-rigid {\color{blue}symmetric} category, then $(-)^*:\Aa^{op}\to\Aa$ is a self-adjoint braided monoidal functor.
\end{proposition}
\begin{proof}
Let $X$ and $Y$ be two objects in $\Aa$. Then we can construct a morphism $\ev_Y\circ (Y^*\ot\ev_X\ot Y):Y^*\ot X^*\ot X\ot Y\to I$. Hence we find, by the universal property, a unique morphism $Y^*\ot X^*\to (X\ot Y)^*$. This allows to endow the functor $(-)^*$ with a monoidal structure. 

To see that the functor is self-adjoint, consider the morphism $\ev_X\circ \gamma_{X,X^*}: X\ot X^*\to I$, where $\gamma$ denotes the braiding of $\Aa$. Then, again by the universal property, there exists a unique morphism $e:X\to X^{**}$ such that $\ev_{X^*}\circ (e\ot X^*)=\ev_X\circ \gamma_{X,X^*}$. 
This morphism can be proven to be both the unit and counit of the adjunction.
\end{proof}

The following is now an easy observation.
\begin{corollary}\colabel{LRFam}
With notation as above, denote $\Cc=\Fam(\Aa^{op})(=\Maf(\Aa)^{op})$ and $\Dd=\Fam(\Aa)$. Then the functor $(-)^*:\Aa^{op}\to\Aa$ induces an adjoint pair $(R,L)$ 
$$R:\Cc\to \Dd,\ R(I,X_i)=(I,X_i^*), \qquad L:\Dd\to \Cc,\ L(J,Y_j)=(J,Y_j^*).$$
where $R$ is a braided monoidal functor.
\end{corollary}

The question whether $(-)^*$ is a liftable functor, was answered in \seref{Mich} in case $\Aa=\Vect$, which -\textcolor{blue}{combined with the fulfillment of the other necessary conditions of \thref{uoadjunction}}- leads to the construction of the finite dual Hopf algebra. A priori, there is no reason why the functor $R$ should be liftable as well. As mentioned before, it is our aim to study a different type of duality than the finite dual. Therefore, we make the following restriction on the base category.

Denote by $\Aa^f$ the full subcategory of $\Aa$ consisting of all rigid objects in $\Aa$. Then $\Aa^f$ is again a braided monoidal category, and moreover rigid by construction. In the sequel, we will call the objects in the associated categories $\Fam(\Aa^f)$ and $\Maf(\Aa^f)$ {\em locally finite} objects\index{object!locally finite}. Furthermore, the functor $(-)^*$ restricts to
an equivalence of categories
\begin{equation}\eqlabel{dagger}
{}^\dagger:(\Aa^f)^{op}\to \Aa^f,\end{equation}
sending any object $A\in\Aa$ to its dual. Now consider $\Cc^f=\Fam((\Aa^f)^{op})(=\Maf(\Aa^f)^{op})$ and $\Dd^f=\Fam(\Aa^f)$. Then the functor \equref{dagger} induces functors 
\begin{equation}\eqlabel{LfRfFam}
R^f:\Cc^f\to \Dd^f,\ R^f(I,X_i)=(I,X_i^\dagger), \qquad L^f:\Dd^f\to \Cc^f,\ L(J,Y_j)=(J,Y_j^\dagger).
\end{equation}
Clearly, $(L^f,R^f)$ is again an equivalence of categories, hence a liftable pair. As a conclusion, we obtain the following result (apply \thref{uoadjunction}):

\begin{theorem}\thlabel{dualityfTHA}
There is a pair of adjoint functors $(\uoL^f,\uoR^f)$ between the categories of \textcolor{blue}{Semi-Hopf monoid-algebras and Semi Hopf monoid-coalgebras,} whose homogeneous parts are rigid objects in $\Aa$. 
More precisely, $(\uoL^f,\uoR^f)$ is an equivalence of categories.
\end{theorem}

As a basic example, one can consider $\Aa=\Vect$ and $\Aa^f=\Vect^f$, the category of finite-dimensional vector spaces.
\textcolor{blue}{One can verify that the two conditions of \thref{uoadjunction} for the lifted adjunction given by \thref{dualityfTHA} to hold at the Hopf algebra level are satisfied, providing us with the following}
\begin{corollary}\colabel{newcoro}
\textcolor{blue}{The pair of adjoint functors $(\uoL^f,\uoR^f)$ from \thref{dualityfTHA} defines an equivalence between the categories of Hopf group-algebras and Hopf group-coalgebras, whose homogeneous parts are finite-dimensional $k$-vector spaces. }
\end{corollary}
Remark that, to an object $(I,X_i)$ in $\Fam(\Aa^f)$, one can associate the vector space $\bigoplus_{i\in I}X_i$, which is no longer finite-dimensional if $I$ is not a finite set. Hence, although the base category consists only of finite-dimensional spaces, our result is applicable to infinite-dimensional spaces as well. A Hopf group-coalgebra $(G,H_g)$ whose homogeneous parts $H_g$ are finite-dimensional for all $g\in G$ was called a finite-dimensional Hopf group-coalgebra in \cite{Tur}. As already mentioned before, we will call such an object a {\em locally finite} Hopf group-coalgebra. A similar definition will be used for Hopf group-algebras. Hence we find that the categories of locally finite Hopf group-coalgebras and locally finite Hopf group-algebras are equivalent (even isomorphic). In particular, if $\ul H=(G,H_g)$ is a locally finite Hopf group-coalgebra with comultiplications $\Delta^{\ul H}_{g,h}:H_{gh}\to H_g\ot H_h$, counit $\epsilon^{\ul H}:H_e\to k$, multiplications $\mu^{\ul H}_g:H_g\ot H_g\to H_g$  and units $\eta^{\ul H}_g:k\to H_g$, then $\ul H^\dagger=(G,H_g^*)$ is a locally finite Hopf group-algebra with multiplications $\mu^{\ul H^\dagger}_{g,h}=(\Delta^{\ul H}_{g,h})^*:H_g^*\ot H_h^*\to H_{gh}^*$, unit $\eta^{\ul H^\dagger}=(\epsilon^{\ul H})^*:k\to H_e^*$, comultiplications $\Delta^{\ul H^\dagger}_{g}=(\mu^{\ul H}_g)^*:H^*_g\to H_g^*\ot H_g^*$ and counits $\epsilon^{\ul H^\dagger}_g=(\eta^{\ul H}_g)^*:H_g^*\to k$. Conversely, if $\ul H=(G,H_g)$ is a locally finite Hopf group-algebra, then $\ul H^\dagger=(G,H_g^*)$ is a locally finite Hopf group-coalgebra and we have that $\ul H^{\dagger\dagger}\cong \ul H$. 

The above duality is in fact a special case of the duality of (regular) multiplier Hopf algebras. 
It is well-known that a finite-dimensional Hopf algebra is a Frobenius algebra. 
For a locally finite Hopf group-algebra, we have the following generalization.
\begin{proposition}
Let $(G,H_g)$ be a locally finite Hopf group-algebra. Then the associated Hopf algebra $H=\bigoplus_{g\in G}H_g$ is a co-Frobenius Hopf algebra. Let us denote $H^\dagger=\bigoplus_{g\in G}H_g^*$. Then we have moreover
$$H^\dagger=\Rat(_{H^*}H^*)=\widehat H,$$ 
where $\widehat H=\{t(-h)~|~h\in H\}\subset H^*$ denotes the reduced dual in the sense of Van Daele (see \cite{VD:duality}).
\end{proposition}
\begin{proof}

Consider any $g\in G$ and take $f\in H^*_g$. Then for any $f'\in H^*$, it is clear that $f'*f\in H^*_g$ (i.e. $f'*f$ is zero if it is evaluated outside $H_g$), since $H_g$ is a subcoalgebra of $H$. Therefore we find that $f'*f=f'(e_{i(1)}) f(e_{i(2)})f_i$, so $f_{[0]}\ot f_{[1]}=f(e_{i(2)})f_i\ot e_{i(1)}$, where $\{(e_i,f_i)\}$ denotes a finite dual basis for $H_g$. This shows that $H_g^*\subset \Rat(_{H^*}H^*)$ for all $g\in G$, in particular $H$ is co-Frobenius as the rational part is non-empty.

By the general theory of co-Frobenius Hopf algebras, we know that $\int_l\ot H\cong \Rat(_{H^*}H^*)$ and this isomorphism is given by $t\ot h\mapsto t(-S(h))$, where $t$ is an integral on $H$. Since the antipode of a co-Frobenius Hopf algebra is bijective, we obtain from this isomorphism that $\Rat(_{H^*}H^*)=\widehat H$.

Furthermore, remark that an integral $t$ on the Hopf subalgebra $H_e$ is also an integral on $H$. Hence, the rational elements of $H^*$ are the functionals of the form $t(-h)$, for $h\in H$ with $t\in H^*$. Since clearly $t(-h)\in \bigoplus_{g\in G}H_g^*$ this finishes the proof.
\end{proof}

Recall that by Van Daele's duality theory, the reduced dual $\widehat H$ of a co-Frobenius Hopf algebra $H$ is a multiplier Hopf algebra and $\widehat{\widehat H}\cong H$ (see \cite{VD:duality}). The observations of the proposition above mean exactly that the duality theory of multiplier Hopf algebras in case of locally finite Hopf group-(co)algebras reduces to \textcolor{blue}{\coref{newcoro}}.

\subsection{The additive free completion under coproducts}\selabel{PFAM}

In order to be able to apply Michaelis' theorem to Hopf group-coalgebras, we need to be able to compute the associated Lie (co)algebras. 
Unfortunately, even if $\Aa$ is additive, $\Fam(\Aa)$ and $\Maf(\Aa)$ are not necessarily additive, hence we cannot compute Lie algebras in these categories. Therefore, we will now consider a variation on these categories having the same objects, but where extra morphisms are added such that the categories become additive. 

Let us first remark that $\Fam(\Aa)$ is, up to equivalence, nothing else than the free completion of $\Aa$ under coproducts, and similarly $\Maf(\Aa)$ is the free completion under products. Recall that the free completion of $\Aa$ under coproducts is a category with coproducts $\ol\Aa$ together with a functor $F:\Aa\to\ol\Aa$ and satisfying the universal property that for any category with coproducts $\Bb$ and functor $G:\Aa\to \Bb$ there exists a unique functor $H:\ol\Aa\to\Bb$ that preserves coproducts and such that $G=H\circ F$. Recall from \cite{Kelly} that the notion of a free completion makes sense in an arbitrary $\Vv$-enriched setting. Explicitly, if $\Aa$ is a $\Vv$-enriched category, then the $\Vv$-enriched free completion under coproducts is the full subcategory of the functor category $[\Aa^{op},\Vv]$ with as objects coproducts of representable functors $YA=\Hom_\Aa(-,A)$, where $Y:\Aa\to [\Aa^{op},\Vv]$ denotes the Yoneda embedding functor. Taking $\Vv=\Set$, we recover exactly the category $\Fam(\Aa)$. Indeed, objects $\coprod_{i\in I} YA_i$ are represented by families of objects in $\Aa$, $(I,A_i)$. Consider a natural transformation $\phi:\coprod_IYA_i\to \coprod_JYB_j$, then this is equivalent to a family of natural transformations $\phi_i:YA_i\to \coprod_JYB_j$. Since coproducts in $\Set$ are disjoint unions, we find, after applying the Yoneda lemma, that $\phi_i\in \bigsqcup_J\Hom_\Aa(A_i,B_j)$, hence there is a particular $j\in J$ such that $\phi_i\in\Hom_\Aa(A_i,B_j)$. Therefore, we obtain the same morphisms for $\Fam(\Aa)$ as described at the start of \seref{Fam}.

Suppose now that $\Aa$ is an additive category and take $\Vv=\Ab$, then we recover an additive version of the $\Fam$-construction, that we will denote as $\PFam$. Objects are the same as here above, i.e. families $(I,A_i)$. A morphism is now a natural transformation $\phi:\coprod_IYA_i\to \coprod_JYB_j$. Since coproducts in $\Ab$ are given by direct sums, we then find that $\phi$ corresponds to a family of morphisms (indexed by $I$) of the form $\phi_i\in\bigoplus_{j\in J}\Hom_\Aa(A_i,B_j)$. Consequently, we find for each $i$ a finite number of corresponding indices $j_{i}\in J$ and morphisms $\phi_{j_i}\in \Hom_\Aa(A_i,B_{j_i})$. 

More precisely, we can describe the additive category $\PFam(\Aa)$ as follows: 
\begin{itemize}
\item The objects in $\PFam(\Aa)$ are the objects of $\Fam(\Aa)$.
\item A morphism $\ul \phi:(I,A_i)\to (J,B_j)$ in $\PFam(\Aa)$ is a pair $\ul \phi=(f,\phi_i)$, where 
$f:I\to P^f(J)$ is a map and where we denote by $P^f(J)$ the finite powerset of $J$, i.e.
$$P^f(J)=\{J'\subset J~|~\# J'<\infty\}$$ 
and $\phi_i:A_i\to \coprod_{j\in f(i)}B_j$. Moreover, we identify
$(f,\phi_i)\equiv(g,\psi_i)$ in $\PFam(\Aa)$ iff the following diagram commutes
\[
\xymatrix{
& \coprod_{j\in f(i)}B_j \ar[dr] \\
A_i \ar[ur]^-{\phi_i} \ar[dr]_-{\psi_i} && \coprod_{j\in f(i)\cup g(i)}B_j \\
& \coprod_{j\in g(i)}B_j \ar[ur]
}
\]
Remark that this last condition means exactly that the corresponding natural transformations $\phi,\psi:\coprod_IYA_i\to\coprod_JYB_j$ coincide.
\item For morphisms $\ul \phi=(f,\phi_i):(I,A_i)\to (J,B_j)$, $\ul\psi=(g,\psi_j):(J,B_j)\to (L,C_\ell)$ the composition 
$\ul\psi\circ \ul\phi=(h,\xi_i): (I,A_i)\to (L,C_\ell)$ is given by 
$$h:I\to P^f(L), i\mapsto \bigcup_{j\in f(i)}g(j),$$
and $\xi_i=\psi_{f(i)}\circ \phi_i$, where $\psi_{f(i)}$ is defined by the universal property of the coproduct as in the following diagram (where the vertical arrows are the canonical injections):
\[
\xymatrix{
&& B_j \ar[d] \ar[rr]^-{\psi_j} && \coprod_{\ell\in g(j)}C_\ell \ar[d] \\
A_i\ar[rr]^-{\phi_i} && \coprod_{j\in f(i)}B_j \ar@{.>}[rr]^-{\psi_{f(i)}=\coprod_{j\in f(i)}\psi_j} && \coprod_{\ell\in \cup_{j\in f(i)} g(j)}C_\ell
}
\]
\end{itemize}
As this will be used in the remaining part of the paper, let us describe explicitly the additive structure of the category $\PFam(\Aa)$:
\begin{itemize}
\item The zero object $\ul 0$ is the unique object with $\emptyset$ as underlying set.
\item Let $(f,\phi_i), (g,\psi_i):(I,A_i)\to (J,B_j)$ be two morphisms in $\PFam(\Aa)$. Then $(f,\phi_i)+(g,\psi_i)=(f\cup g,\phi_i+\psi_i)$
where $(f\cup g)(i)=f(i)\cup g(i)$ and 
\[
\xymatrix{
A_i 
\ar[rr]^-{\phi_i+\psi_i} && \coprod_{j\in f(i)\cup g(i)}B_j \\
}
\]
\item For a morphism $(f,\phi_i)$ in $\PFam(\Aa)$, the inverse morphism $-(f,\phi_i)$ is given by $(f,-\phi_i)$.
\end{itemize}

If $\Aa$ is braided (symmetric) then $\PFam(\Aa)$ is again braided (resp. symmetric) with the same braiding as the one we defined on $\Fam(\Aa)$. 
Since $\PFam(\Aa)$ has the same objects as $\Fam(\Aa)$ and as it has more morphisms, there is an obvious embedding functor $\sf P:\Fam(\Aa)\to \PFam(\Aa)$. 

In an obvious dual way, one defines $\PMaf(\Aa)$. From the universal properties of free completion, one then arrives at the following result.

\begin{proposition}\prlabel{FamvsPFam}
With notation as above, if $\Aa$ is an additive symmetric monoidal category, then $\PFam(\Aa)$ and $\PMaf(\Aa)$ are additive symmetric monoidal categories as well. Furthermore, we have a diagram of functors 
\[
\xymatrix{
\Fam(\Aa) \ar[d]_{\sf P} \ar[rr]^{\sf F} && \Set && \Maf(\Aa) \ar[ll]_{\rff} \ar[d]^{\rpp} \\
\PFam(\Aa)  \ar[rr]^{\sf S} && \Aa && \PMaf(\Aa) \ar[ll]_{\rss} 
}
\]
where $\sf F$ and $\rf$ are strong monoidal forgetful functors, $\sf P$ and $\rp$ are strong monoidal functors, and $\sf S$ and $\rs$ are additive monoidal functors, sending $(I,A_i)$ to $\coprod_{i\in I}A_i$.
\end{proposition}

Extending our point of view from the category of families $\Fam(\Aa)$ to the additive category $\PFam(\Aa)$ forces the existence of many new types of Hopf algebras, as we have introduced new morphisms that can serve as multiplication, comultiplication, etc. We will not describe these new classes of Hopf algebras in this note. Thanks to the observations from \prref{FamvsPFam} we know that all Hopf algebras in $\Fam(\Aa)$ are still Hopf algebras in $\PFam(\Aa)$. Indeed, the braided strong monoidal functor $\sf P$ preserves Hopf algebras. Since the Hopf algebras in $\Fam$ are objects of main importance, we will focus on these in the sequel. For the same reason, we will not describe all Lie algebras and Lie coalgebras in these categories. Let us restrict to the following observation, which is sufficient for our needs. 
\begin{lemma}\lelabel{LiealgPMaf}
Let $\{(L_i,\lie_{i})\}_{i\in I}$ be a family of Lie algebras in $\Aa$, then $(I,L_i)$ is a Lie algebra in $\PMaf(\Aa)$. Dually, let $\{(C_i,\Upsilon_{i})\}_{i\in I}$ be a family of Lie coalgebras in $\Cc$, then $(I,C_i)$ is a Lie coalgebra in $\PFam(\Aa)$.
\end{lemma}
\begin{proof}
We know that any set $I$ possesses a unique coalgebra structure in $\Set$, with comultiplication $d:I\to I\times I,\ d(i)=(i,i)$. We define the bracket $\ul\lie=(d,\lie_i)$ on $(I,L_i)$, then it is easily checked that this defines a Lie algebra in $\PMaf(\Aa)$. 
\end{proof}

In principle, we now have all the tools to obtain a version of Michaelis' theorem for Turaev's Hopf group-coalgebras. It suffices to verify that the pair $(L^f,R^f)$ is Lie-liftable, apply our main result \thref{Michaelis} to the additive categories $\PFam(\Aa)$ and $\PMaf(\Aa)$ described above, and implement it for Hopf group-(co)algebras, which are particular Hopf algebras in these categories. In the next Section, we will work this out in a more explicit way.

\subsection{Michaelis' theorem for Hopf group-coalgebras}

As we have two kinds of Hopf algebras, namely Hopf group-algebras and Hopf group-coalgebras, we can prove two versions of Michaelis' theorem in the setting of Turaev's Hopf algebras.

Firstly, we compute (see \thref{MichTur1} below) the primitive elements of a Hopf group-algebra, which turn out to be just the primitive elements of the degree 1 homogeneous part, which is a (classical) Hopf sub-algebra. In a similar way, the indecomposables of a Hopf group-coalgebra are just the indecomposables of the degree 1 homogeneous part, which is again a Hopf sub-algebra. In the previous section, we found that the pair $(L^f,R^f)$ between the categories of families of finite-dimensional vector spaces is liftable (and Lie-liftable). Applying these results, we would obtain that for a locally finite Hopf group-algebra $H$, we have $P(H_e^*)\cong Q(H_e)^*$. However, as $H_e$ is a usual Hopf algebra, even for a Hopf group-algebra that is not locally finite, we can also apply Michaelis' classical theorem to $H_e$ and obtain a slightly more general result, valuable for all Hopf group-algebras.
All this leads us to the following theorem.

\begin{theorem}\thlabel{MichTur1}
\begin{enumerate}[(i)]
\item
Let $(G,H_g)$ be a Hopf group-algebra  (i.e.\ a Hopf algebra in $\Fam(\Vect)$), then the primitive elements of $H=\bigoplus_{g\in G}H_g$ are exactly the primitive elements of $H_e$, hence they form a classical Lie algebra.
\item
Dually, if $(G,H_g)$ is a Hopf group-coalgebra, the indecomposables of $(G,H_g)$ are exactly the indecomposables of $H_e$, hence they form a classical Lie coalgebra.
\item
Consequently, given a Hopf group-coalgebra $(G,H_g)$, we have the following isomorphism of Lie algebras
$$P(H_e^\circ)\cong Q(H_e)^*.$$
\end{enumerate}
\end{theorem}
\begin{proof}
\ul{(i)}
Take $H=\bigoplus_{g\in G}H_g$, then $H$ is a usual Hopf algebra, graded as an algebra. We can compute the primitive elements of $H$. These are elements $x\in H$ such that $\Delta(x)=1\ot x+x\ot 1$. We can write $x=\sum_{i\in I} x_i$ with $I\subset G$ and $x_i\in H_i$. For each $x_i$ we have that $\Delta(x_i)\in H_i\ot H_i$. Hence $\Delta(x)\in \bigoplus_{i\in I} H_i\ot H_i=: X$. On the other hand, $1\in H_e$. Hence $1\ot x\in H_e\ot (\bigoplus_{i\in I} H_i)=\bigoplus_{i\in I} (H_e\ot H_i)$ and hence $1\ot x+x\ot 1\in (\bigoplus_{i\in I} (H_e\ot H_i))\oplus (\bigoplus_{i\in I} (H_i\ot H_e))=:Y$. The spaces $X$ and $Y$ only have a non-zero intersection if $e\in I$ and in this case the intersection is exactly $H_e\ot H_e$. By conclusion, the only primitive elements of $H$ are the primitive elements of the Hopf algebra $H_e$ concentrated in degree $e$, i.e. $P(H)=P(H_e)$. 
\item \ul{(ii)}. Follows from dual arguments.
\item \ul{(iii)}. Follows now directly from the classical version of Michaelis' theorem. 
\end{proof}

The second Michaelis-type theorem for Turaev's Hopf algebras requires a computation of the primitive elements of Hopf group-coalgebras. This leads to the following definition.

\begin{definition}
Let $\ul H= (G,H_g)$ be a Hopf group-coalgebra  and $g\in G$. A \textit{$g$-primitive element}\index{element!$g$-primitive} of $\ul H$ is an element $x\in H_g$ such that there exists an element $(x_h)_{h\in G}\in \prod_{h\in G}H_h$ with $x_g=x$, such that for all $h,h'\in G$, with $g=hh'$, 
$$\Delta_{h,h'}(x)=1_{h}\ot x_{h'}+x_{h}\ot 1_{h'}.$$
The set of all $g$-primitive elements of $\ul H$ will be denoted by $P_g=P_g(\ul H)$.
\end{definition}

Remark that the $e$-primitive elements of $\ul H$ are (classical) primitive elements of the (usual) Hopf algebra $H_e$.

\begin{lemma}\lelabel{gprimitive}
Let $\ul H=(G,H_g)$ be a Hopf group-coalgebra, $g\in G$ and $x$ be a $g$-primitive element. Denote by $(x_g)_{g\in G}\in\prod_{g\in G}H_g$ the element associated to $x$ (as in the definition of a $g$-primitive element). Then the following assertions hold
\begin{enumerate}[(i)]
\item $x_h$ is an $h$-primitive element for all $h\in H$, more precisely, $\Delta_{h'h''}(x_h)=1_{h'}\ot x_{h''}+x_{h'}\ot 1_{h''}$ for all $h',h''\in G$ such that $h=h'h''$;
\item $\epsilon(x_e)=0$.
\end{enumerate}
\end{lemma}

\begin{proof}
\ul{(i)}. Take $h,h',h''\in G$ as in the statement. Then there exists an $\ell\in G$ such that $h\ell=g$. By coassociativity in the Hopf group-coalgebra, we find
\begin{eqnarray*}
(\Delta_{h',h''}\ot H_\ell)\circ \Delta_{h,\ell}(x)&=& (\Delta_{h',h''}\ot H_\ell)(1_h\ot x_\ell + x_h\ot 1_\ell)\\
&=&1_{h'}\ot 1_{h''}\ot x_\ell + \Delta_{h',h''}(x_h)\ot 1_\ell\\
=(H_{h'}\ot \Delta_{h'',\ell})\circ\Delta_{h',h''\ell}(x)&=& (H_{h'}\ot \Delta_{h'',\ell})(1_{h'}\ot x_{h''\ell}+x_{h'}\ot 1_{h''}\ell)\\
&=&1_{h'}\ot \Delta_{h'',\ell}(x_{h''\ell})+x_{h'}\ot 1_{h''}\ot 1_\ell
\end{eqnarray*}
From this equality, we can deduce that there exists an element $y\in H_{h''}$ such that
$$\Delta_{h',h''}(x_h)= x_{h'}\ot 1_{h''}+1_{h'}\ot y \quad {\rm and}\quad 
\Delta_{h'',\ell}(x_{h''\ell})=1_{h''}\ot x_\ell + y\ot 1_\ell $$
Using a symmetric argument with $\ell'\in G$ such that $\ell'h=g$, we find that $y=x_{h''}$, and the statement is proved.
\item \ul{(ii)}. By the counit condition, we know that $x=(\epsilon\ot H_g)\circ\Delta_{e,g}(x)$. Hence, 
$$x=\epsilon(x_e)1_g+\epsilon(1_e)x.$$
Since $\epsilon(1_e)=1$, we conclude that $\epsilon(x_e)=0$.
\end{proof}

\begin{theorem}
Let $(G,H_g)$ be a Hopf group-coalgebra (i.e. a Hopf algebra in $\PMaf$), then $(G,P_g)$ is a collection of Lie algebras (hence $(G,P_g)$ is a Lie algebra in $\PMaf$ by \leref{LiealgPMaf}).
\end{theorem}

\begin{proof}
The set of all $g$-primitive elements of $\ul H$ clearly forms a subspace of $H_g$, because of the linearity of the coproduct and tensor product. Now, fix $g\in G$ and take $x, y\in P_g$ and denote $(x_h)_{h\in G}$ and $(y_h)_{h\in G}$ for the associated elements in $\prod_{h\in G}H_h$. Let us check that the commutator $[x,y]=xy-yx$ is still an element of $P_g$, with $([x_h,y_h])_{h\in G}$ the associated element in $\prod_{h\in G}H_h$. 
We compute:
\begin{eqnarray*}
\Delta_{h,h'}(xy-yx)&=&\Delta_{h,h'}(xy)-\Delta_{h,h'}(yx)=\Delta_{h,h'}(x)\Delta_{h,h'}(y)-\Delta_{h,h'}(y)\Delta_{h,h'}(x)=\\
&=& 1_{h}\ot(x_{h'}y_{h'})+{x_{h} \ot y_{h'}} + {y_{h}\ot x_{h'}}+ ({x_{h}y_{h}})\ot 1_{h'} \\
&& - (1_{h}\ot({y_{h'}x_{h'}})+{y_{h} \ot x_{h'}} + {x_{h}\ot y_{h'}}+ ({y_{h}x_{h}})\ot 1_{h'})\\
&=& 1_{h}\ot ({x_{h'}y_{h'}}-y_{h'}x_{h'})+({x_{h}y_{h}}-{y_{h}x_{h}})\ot 1_{h'}, 
\end{eqnarray*}
where we used the very definition of $x$ and $y$ being $g$-primitive as well as the fact that $\Delta_{h,h'}$ is a morphism of algebras.
\end{proof}

Recall that if $\ul H=(G,H_g)$ is a Hopf algebra in $\Fam(\Vect)$, then $H=\oplus_{g\in G}H_g$ is a $k$-Hopf algebra. Hence, we can compute the space of indecomposables of $H$, i.e. the space $Q=\ker\epsilon/(\ker\epsilon)^2$, where $\epsilon=\sum_{g\in G}\epsilon_g$. We know from the general theory (see \seref{Mich}) that $Q$ is a Lie coalgebra, and that there is a canonical morphism
$$\pi:H\to Q,\ x\mapsto [x-\epsilon(x)1],$$
where $[y]$ denotes the equivalence class of $y(\in\ker \epsilon)$ in $\ker\epsilon/(\ker\epsilon)^2$.
This leads us to the following definition.

\begin{definition}
Let $\ul H=(G,H_g)$ be a Hopf group-algebra (i.e.\ a Hopf algebra in $\Fam(\Vect)$). Then the \textit{$g$-indecomposables}\index{element!$g$-indecomposable} of $\ul H$, denoted by $Q_{g}=Q_g(\ul H)$, are the elements of the subspace of $Q$ generated by the image of $\pi_g=\pi|_{H_g}$.
\end{definition}

\begin{lemma}\lelabel{propgindec}
Let $(G,H_g)$ be a Hopf algebra in $\Fam(\Vect)$, then for $x\in H_h$ and $y\in H_{h'}$ such that $hh'=g$, we have
$$\pi_g(xy)=\pi_h(x)\epsilon_{h'}(y)+\epsilon_h(x)\pi_{h'}(y).$$
\end{lemma}

\begin{proof}
First remark that $Q$ consists of classes of elements $[x]$ with $x\in H$ that satisfy $\epsilon(x)=0$, and such that $[x]=[y]$ if and only if $x-y=\sum_i z_iz'_i$ with $\epsilon(z_i)=\epsilon(z'_i)=0$.

Now take $x$ and $y$ as in the statement of the Lemma. 
Then 
\begin{eqnarray*}
\pi_g(xy)&=&[xy-\epsilon_g(xy)1]\\
\pi_h(x)\epsilon_{h'}(y)+\epsilon_h(x)\pi_{h'}(y)&=& \pi(x\epsilon_{h'}(y)+\epsilon_h(x)y)\\
&=& [x\epsilon_{h'}(y)+\epsilon_h(x)y-2\epsilon_h(x)\epsilon_{h'}(y)1]
\end{eqnarray*}
Since $xy-x\epsilon_{h'}(y)-\epsilon_h(x)y+\epsilon_h(x)\epsilon_{h'}(y)1
=(y-\epsilon_{h'}(y)1)(x-\epsilon_h(x)1)
\in (\ker\epsilon)^2$, we find that $\pi_g(xy)=\pi_h(x)\epsilon_{h'}(y)+\epsilon_h(x)\pi_{h'}(y)$.
\end{proof}

\begin{theorem}
Let $(G,H_g)$ be a Hopf algebra in $\Fam(\Vect)$, then $(G,Q_g)$ is a collection of Lie coalgebras.
\end{theorem}

\begin{proof}
We know (see e.g.\ 
\cite[Page 7]{Mich}) that the indecomposables $Q$ of $H=\oplus_{g\in G}H_g$ form a Lie coalgebra. The Lie co-bracket $\Upsilon_Q$ is induced by the commutator Lie co-bracket $\Upsilon=\Delta-\tau_{H\ot H}\circ \Delta$ on $H$, in the sense that the following diagram commutes.
\[
\xymatrix{
H \ar[rr]^-{\pi} \ar[d]^{\Upsilon} &&Q \ar[d]^{\Upsilon_Q} \\
H\ot H \ar[rr]^{\pi \ot \pi} && Q\ot Q
}
\]
Furthermore, the restriction of $\Delta$ to $H_g$ is exactly $\Delta_g$, whose image lies inside $H_g\ot H_g$, hence $\Upsilon(H_g)\subset H_g\ot H_g$. Therefore, since $Q_g$ is defined as $\pi(H_g)$ we find that 
$\Upsilon_Q(Q_g)=\Upsilon_Q(\pi(H_g))=(\pi\ot\pi)(\Upsilon(H_g))\subset (\pi\ot\pi)(H_g\ot H_g)\subset Q_g\ot Q_g$.
So the restriction of $\Upsilon_Q$ on $Q_g$ gives us a well-defined Lie co-bracket $\Upsilon_g$ on $Q_g$, that turns $Q_g$ into a Lie coalgebra.
\end{proof}

The main theorem of this Section is the following. It is a version of Michaelis' theorem for Hopf group-(co)algebras.

\begin{theorem}\thlabel{MichTur2}
If $\ul H=(G,H_g)$ is a locally finite Hopf group-algebra (i.e. a Hopf algebra $\Fam({\Vect}^{f}(k))$), there is a natural isomorphism between the collections of Lie algebras
$$(G,P_{g}(\ul H^\dagger))\cong (G,Q_{g}(\ul H)^*).$$
\end{theorem}

\begin{proof}
Fix $g\in G$. Let us put $Q_{g}=Q_{g}(\ul H)$ and $P_{g}=P_{g}(\ul H^\dagger)$. 
Recall that there exists a surjective map $\pi_g:H_g\to Q_g$. Dualizing this map, we obtain the injective morphism $\pi_g^*:Q_g^*\to H_g^*$. Our aim is to show that all elements in the image of $\pi_g^*$ satisfy the defining condition to be in $P_g$. We then will have found an injective map $\alpha:Q_g^*\to P_g$, defined as $\alpha(f)=f\circ \pi_g$, for all $f\in Q_g^*$. By construction, we have a canonical injection $i_g:Q_g\to Q$, hence there is a surjective morphism $i^*_g:Q^*\to Q^*_g, i^*_g(f')=f'\circ i_g=f'|_{Q_g}$ for $f'\in Q^*$. 
By the surjectivity of the map $i^*_g$, we can write any element $f$ in $Q_g^*$ as $f=i^*_g(f')$ for some $f'\in Q^*$. 
So, taking any $f'\in Q^*$, we find that $i^*_h(f')\in Q^*_h$ and $f_h:=\pi^*_h\circ i^*_h(f')\in H_h^*$ for all $h\in G$. We obtain in this way an element $(f_h)_{h\in G}\in \coprod_{h\in G}H_h^*$ and will prove that $f_g\in P_g$. Indeed, take $h,h'\in G$ such that $hh'=g$. By definition of $\ul H^\dagger$, we have that $\Delta_{h,h'}^{\ul H^\dagger}(a)=a\circ \mu^{\ul H}_{h,h'}$, for all $a\in H_{g}^*$. Hence, we find for all $x\in H_h$ and $y\in H_{h'}$, 
\begin{eqnarray*}
f_g(xy)&=& f'(\pi_g(xy))
=f'(\pi_h(x)\epsilon_{h'}(y)+\epsilon_{h}(x)\pi_{h'}(y))\\
&=& f'(\pi_h(x))\epsilon_{h'}(y)+\epsilon_{h}(x) f'(\pi_{h'}(y))\\
&=& f_h(x)\epsilon_{h'}(y)+\epsilon_{h}(x)f_{h'}(y).
\end{eqnarray*}
Here we used \leref{propgindec} in the second equality.
Hence, $\Delta_{h,h'}(f_g)=(f_h \ot \epsilon_{h'})+ (\epsilon_h\ot f_{h'})$, so $f_g\in P_g$ and $\alpha(f)\in P_g$ for all $f\in Q_g$. 

Conversely, take $p\in P_g(H^*)$, then there exists an element $\ul p=(p_g)_{g\in G}\in\prod_{g\in G} H_g^*$ such that $p_g=p$ and $\Delta_{h,h'}^{\ul H^\dagger}(p)=\epsilon_h\ot p_{h'}+p_h\ot \epsilon_{h'}$ for all $h,h\in G$ with $g=hh'$. Moreover, there is a canonical isomorphism $\prod_{g\in G} H_g^* \cong (\oplus_{g\in G}H_g)^*$. Under this isomorphism, we can view $\ul p$ as a functional on $H=\oplus_{g\in G}H_g$ by putting for any $x=\sum_{g\in G} x_g\in \oplus_{g\in G}H_g$ (with only a finite number of terms $x_g$ being non-zero), $\ul p(x)=\sum_g p_g(x_g)$. Consider now $x_{h'}\in \ker\epsilon_{h'}$ and $x_{h''}\in \ker\epsilon_{h''}$, i.e. $x=x_{h'}x_{h''}\in(\ker\epsilon)^2$. Then 
\begin{eqnarray*}
\ul p(x)&=&p_{h'h''}(x_hx_{h'})=p_{h'h''(1)}(x_{h'})\ul p_{h'h''(2)}(x_{h''})\\
&=&p_{h'}(x_{h'})\epsilon_{h''}(x_{h''})+\epsilon_{h'}(x_{h'})\ul p_{h''}(x_{h''})=0.
\end{eqnarray*}
where we applied \leref{gprimitive} in the third equality. We obtain that $\ul p$ is well-defined (as being an element in $Q^*$). By restriction, we then find that $\ul p$ is also well-defined considered as being an element in $Q_p^*$. 
In this way, we have found a well-defined map
$\beta:P_g\to Q_g^*, \beta(p)(\pi_g(x))=\ul p(x)$, for any $\pi_g(x)\in Q_g$.

Let us show that $\beta$ is injective. Suppose that $\beta(p)=0$, then $\beta(p)(\pi_g(x))=\ul p(x)=0$ for all $\pi_g(x)$ in $Q_g$. Since $\beta(p)$ is well-defined, this means 
that $\ul p(x)=0$ for all $x\in H_g$. Hence $p=p_g=0$ and $\beta$ is injective.

Since both $\alpha$ and $\beta$ are injective morphisms between finite-dimensional vector spaces, $P_g$ and $Q_{g}^*$ are isomorphic (through $\alpha$ and $\beta$). Let us finish by showing that $\alpha:Q_g^*\to P_g, \alpha(f)=f\circ\pi_g$ is a Lie algebra morphism.
Take any $f,g\in Q_g^*$ and $q=\pi_g(x)\in Q_g$. Recall that 
$\Upsilon_g(\pi_g(x))=\pi_g(x_{(1)})\ot \pi_g(x_{(2)})- \pi_g(x_{(2)})\ot \pi_g(x_{(1)})$.
 Therefore,
\begin{eqnarray*}
\alpha(\Upsilon^*(f\ot g))(x)&=&\Upsilon^*(f\ot g)(\pi_g(x))\\
&=&f(\pi_g(x_{(1)}))g(\pi_g(x_{(2)}))- f(\pi_g(x_{(2)}))g(\pi_g(x_{(1)}))\\
&=&\alpha(f)(x_{(1)})\alpha(g)(x_{2})-\alpha(f)(x_{(2)})\alpha(g)(x_{(1)})\\
&=&[\alpha(f),\alpha(g)](x).
\end{eqnarray*}
So $\alpha(\Upsilon^*(f\ot g))=[\alpha(f),\alpha(g)]$, and $\alpha$ is indeed a Lie algebra morphism.
\end{proof}

\begin{remark}
Although not stated explicitly this way, \thref{MichTur2} follows from our general result \thref{Michaelis} applied to the adjoint pair $(L^f,R^f)$. To see this, one needs to observe that the primitive elements $(G,P_g)$ and indecomposables $(G,Q_g)$ are indeed computed as respectively the following equalizer in $\Maf(\Vect^f)$ and coequalizer in $\Fam(\Vect^f)$
\[
\xymatrix{
(G,P_g)\ar[rr] && (G,H_g) \ar@<.5ex>[rr]^-{\ul{\Delta}} \ar@<-.5ex>[rr]_-{\ul\eta\ot\ul H+\ul H\ot \ul\eta} && (G\times G,H_g\ot H_h) \\
(G\times G,H_g\ot H_h) \ar@<.5ex>[rr]^-{\ul{\mu}} \ar@<-.5ex>[rr]_-{\ul\epsilon\ot\ul H+\ul H\ot \ul\epsilon} &&
(G,H_g) \ar[rr] && (G,Q_g)
}
\]
In fact, if one computes this equalizer and coequalizer more explicitly, they come down to a multiple pullback and pushout in $\Vect$, such that each $P_g$ and $Q_g$ are characterized exactly by the equalities given in \leref{gprimitive}(i) and \leref{propgindec} respectively.
\end{remark}

\subsection*{Closing remarks}
We proved a version of Michaelis' theorem in the setting of symmetric additive monoidal categories; \thref{Michaelis}. We showed that this abstract version of the theorem allows to recover the original result of Michaelis (\seref{Mich}), as well as a super and a Hom-version.
It also can be applied to other situations. In particular, we found a version of Michaelis' theorem for Turaev's Hopf algebras which are locally finite (\thref{MichTur2}). In this way, we have a version of Michaelis' theorem for a special class of multiplier Hopf algebras and co-Frobenius Hopf algebras.
It would be interesting to see whether a more general version of Michaelis' theorem holds for larger classes of multiplier Hopf algebras, and (eventually) for locally compact quantum groups; or for Lie algebroids and Hopf algebroids. This is a direction for future research. 

\subsubsection*{Acknowledgement}

The authors would like to thank Alessandro Ardizzoni for his useful remarks on a preliminary version of this paper{\color{blue}, and for pointing out the missing condition of symmetry in \prref{prerigid}, cf.\ \cite[remark 3.5]{AGM}.}\\
The authors also want to thank the referee for the very useful remarks regarding the $\PFam$ construction, that improved the presentation of \seref{PFAM}.\\
\textcolor{blue}{Finally, the authors would like to express their gratitude to H.-E. {\color{blue}Porst} for pointing out an inaccuracy in Theorem 2.7 of the published version of this paper (as mentioned in a footnote preliminary to the actual version).}

\end{document}